\def\todaysdate{13\textsuperscript{th} May 2022}
\documentclass[reqno,a4paper]{article}
\usepackage{etex}
\usepackage[T1]{fontenc}
\usepackage[utf8]{inputenc}
\usepackage{lmodern}
\usepackage{ amssymb }
\usepackage{subcaption}
\usepackage{comment}

\usepackage{amssymb,amsmath,amsthm}
\usepackage{thmtools,xcolor}
\usepackage{units}
\usepackage{graphicx}
\usepackage[all]{xy}
\usepackage{tikz}
\usetikzlibrary{arrows,calc,positioning,decorations.pathreplacing}
\usepackage{relsize}
\usepackage{mhsetup}
\usepackage{mathtools}
\usepackage{float}
\usepackage{stmaryrd}
\usepackage{tocloft}
\usepackage{paralist} 
\usepackage[left=3cm,right=3cm,top=3cm,bottom=3cm]{geometry} 
\usepackage[bottom]{footmisc}
\usepackage[colorlinks,citecolor=blue,linkcolor=blue,urlcolor=blue,filecolor=blue,breaklinks]{hyperref}
\usepackage{footnotebackref}
\usepackage[format=plain,indention=1em,labelsep=quad,labelfont={up},textfont={footnotesize},margin=2em]{caption}
\usepackage[mathscr]{euscript}
\usepackage{accents}

\definecolor{lightblue}{rgb}{0.8,0.8,1}

\numberwithin{equation}{section}
\numberwithin{figure}{section}

\definecolor{vdarkred}{rgb}{0.7,0,0}
\declaretheoremstyle[
  spaceabove=\topsep,
  spacebelow=\topsep,
  headpunct=,
  numbered=no,
  postheadspace=1ex,
  headfont=\color{vdarkred}\normalfont\bfseries,
  bodyfont=\normalfont\itshape,
]{colored}
\declaretheoremstyle[
  spaceabove=\topsep,
  spacebelow=\topsep,
  headpunct=,
  numbered=no,
  postheadspace=1ex,
  headfont=\normalfont\bfseries,
  bodyfont=\normalfont\itshape,
]{italic}
\declaretheoremstyle[
  spaceabove=\topsep,
  spacebelow=\topsep,
  headpunct=,
  numbered=no,
  postheadspace=1ex,
  headfont=\normalfont\bfseries,
  bodyfont=\normalfont\upshape,
]{upright}

\declaretheorem[style=italic,name=Theorem,numbered=yes,numberwithin=section]{thm}
\declaretheorem[style=italic,name=Lemma,numbered=yes,numberlike=thm]{lem}
\declaretheorem[style=italic,name=Proposition,numbered=yes,numberlike=thm]{prop}
\declaretheorem[style=italic,name=Corollary,numbered=yes,numberlike=thm]{coro}

\declaretheorem[style=upright,name=Definition,numbered=yes,numberlike=thm]{defn}
\declaretheorem[style=upright,name=Remark,numbered=yes,numberlike=thm]{rmk}

\declaretheorem[style=upright,name=Notation,numbered=yes,numberlike=thm]{notation}

\makeatletter
\renewcommand*{\@seccntformat}[1]{\upshape\csname the#1\endcsname.\hspace{1ex}}
\renewcommand*{\section}{\@startsection{section}{1}{\z@}%
	{2.5ex \@plus 1ex \@minus 0.2ex}%
	{1.5ex \@plus 0.2ex}%
	{\normalfont\normalsize\bfseries}}
\renewcommand*{\subsection}{\@startsection{subsection}{2}{\z@}%
	{2.5ex \@plus 1ex \@minus 0.2ex}%
	{-1.5ex \@plus -0.2ex}%
	{\normalfont\normalsize\bfseries}}
\renewcommand*{\subsubsection}{\@startsection{subsubsection}{3}{\z@}%
	{2.5ex \@plus 1ex \@minus 0.2ex}%
	{-1.5ex \@plus -0.2ex}%
	{\normalfont\normalsize\bfseries}}
\renewcommand*{\paragraph}{\@startsection{paragraph}{4}{\z@}%
	{2.5ex \@plus 1ex \@minus 0.2ex}%
	{-1.5ex \@plus -0.2ex}%
	{\normalfont\normalsize\bfseries}}
\renewcommand*{\subparagraph}{\@startsection{subparagraph}{5}{\z@}%
	{2.5ex \@plus 1ex \@minus 0.2ex}%
	{-1.5ex \@plus -0.2ex}%
	{\normalfont\normalsize\slshape}}
\makeatother

\newcommand{\cs}{\mathscr S_{n}^N}
\newcommand{\ct}{\mathscr T_{n}^N}

\newcommand{\N}{\mathbb N}
\newcommand{\Sy}{\Sigma}
\newcommand{\Z}{\mathbb Z}
\newcommand{\Diff}{\mathrm{Diff}}

\newcommand{\card}{\mathrm{card}}

\newcommand{\Co}{\mathrm{Col}}

\definecolor{dgreen}{RGB}{0,150,0}

\newcommand{\incl}[3][right]%
{%
\draw[<-,>=#1 hook] #2 to ($ #2!0.5!#3 $);
\draw[->,>=stealth'] ($ #2!0.5!#3 $) to #3;%
}
\newcommand{\inclusion}[5][right]%
{%
\draw[<-,>=#1 hook] #4 to ($ #4!0.5!#5 $) node[#2,font=\small]{#3};
\draw[->,>=stealth'] ($ #4!0.5!#5 $) to #5;%
}

{\begin{compactitem}

}%
{\end{compactitem}}
{\begin{compactitem}[#1]

}%
{\end{compactitem}}
{\begin{compactdesc}

}%
{\end{compactdesc}}

\renewcommand{\geq}{\geqslant}
\renewcommand{\leq}{\leqslant}
\renewcommand{\footnoterule}{%
  \kern -3pt
  \hrule width \textwidth height 0.4pt
  \kern 2.6pt
}

\definecolor{dgreen}{RGB}{0,150,0}

\begin{document}
\title{\vspace{-14mm} \Large\bfseries $U_q(sl(2))$-quantum invariants from an intersection of two Lagrangians in a symmetric power of a surface}
\author{ \small Cristina Anghel \quad $/\!\!/$\quad \todaysdate \vspace*{-8ex}}
\date{}
\vspace{-15mm}
\maketitle
{
\makeatletter
\renewcommand*{\BHFN@OldMakefntext}{}
\makeatother
}
\vspace{-9mm}
\begin{abstract}
In this paper we show that coloured Jones and coloured Alexander polynomials can both be read off from the same picture provided by two Lagrangians in a symmetric power of a surface.  More specifically, the  $N^{th}$ coloured Jones and $N^{th}$ coloured Alexander polynomials are specialisations of a graded intersection between two explicit Lagrangian submanifolds in a symmetric power of the punctured disc. The graded intersection is parametrised by the intersection points between these Lagrangians, graded in a specific manner using the diagonals of the symmetric power.  As a particular case, we see the original Jones and Alexander polynomials as two specialisations of a graded intersection between two Lagrangians in a configuration space, whose geometric supports are Heegaard diagrams. 
\end{abstract}
\vspace{-5mm}
{\tableofcontents}

\vspace{-3mm}

\section{Introduction}\label{introduction}
\vspace{-2mm}
Coloured Jones and coloured Alexander polynomials are two sequences of quantum invariants coming from the representation theory of the quantum group $U_q(sl(2))$. More specifically, the quantum group with generic $q$ gives the family of coloured Jones polynomials $\{J_N(L,q) \in \mathbb Z[q^{\pm 1}]\}_{N \in \mathbb N}$ associated to a link $L$, recovering the original Jones polynomial at the first term (corresponding to $N=2$). The same quantum group at roots of unity leads to the sequence of coloured Alexander polynomials $\{\Phi_N(L,\lambda)\}$ (or ADO invariants \cite{ADO}), which are non-semisimple invariants. This family recovers the Alexander polynomial at its first term. These invariants came originally from representation theory (\cite{J},\cite{RT}) but there are important open questions predicting that their asymptotics with respect to the colour contain geometric information. At the same time, there is an active research area concerning relations between the categorifications of the Alexander polynomial and Jones polynomial provided by Khovanov homology and knot Floer homology (\cite{SM}, \cite{M1},\cite{Ras},\cite{D},\cite{KWZ}).
The first topological model was constructed by Bigelow \cite{Big} for the Jones polynomial, as a graded intersection of submanifolds in configuration spaces, using Lawrence representations \cite{Law}. 

In recent years, we studied topological models for the sequences of $U_q(sl(2))$-quantum invariants, and we showed in \cite{Cr2} that $J_N(L,q)$ and $\Phi_N(L,\lambda)$ are specialisations of a {\em state sum of graded Lagrangian intersections} in a configuration space in the punctured disc. The main purpose of this paper is to show a globalised result, which {\em uses a single graded intersection between two submanifolds} and does {\em not involve any state sum}. 

Our main result is that the $N^{th}$ coloured Jones and $N^{th}$ coloured Alexander polynomials of the closure of a braid with $n$ strands can be read off from a {\em graded intersection between two Lagrangians in the $(n-1)(N-1)^{th}$ symmetric power of a punctured disc}. The feature of this formula is that the Lagrangian submanifolds are explicit and their geometric support in the surface is a multi-pointed Heegaard diagram for a knot. Also, the colour $N$ of the invariants can be seen directly in this geometric picture by the number of particles chosen on a fixed geometric support given by the diagram. For the particular case when $N=2$, this gives a {\em picture which provides a framework for constructing geometric categorifications} for Jones and Alexander polynomials and a further perspective to relate two such theories.

\subsection{Main result}
We look at links as being the closure of braids. Let us fix $n\in \N$ the number of the strands of such a braid and $N \in \N$ the colour of our quantum invariants. We start with the $(3n-1)$-punctured disc and consider $\Sigma^{n,N}$ to be its $(n-1)(N-1)$-symmetric power. 
The construction of the Lagrangians will be done by drawing a collection of $(n-1)$ arcs/ circles in the punctured disc, considering products of $(N-1)$-symmetric powers on each of them and taking their quotient to the symmetric power. We call this  set of curves in the punctured disc the ``geometric support'' of the submanifold. Moreover, we will use graded intersections, and the grading procedure will be prescribed by certain curves in the punctured disc which connect the two submanifolds.

Following this procedure, we define two Lagrangians $\mathscr S_{n}^N$ and $\mathscr T_{n}^N$ in the symmetric power $\Sigma^{n,N}$, given by the collections of red arcs and green circles respectively:
$$\hspace{-10mm}{\color{red}(\beta_n \cup \mathbb I_{2n-1})\mathscr S_{n}^N}\cap{\color{dgreen}\mathscr T_{n}^N}$$
\vspace{-10mm}
\begin{figure}[H]
\centering
\includegraphics[scale=0.4]{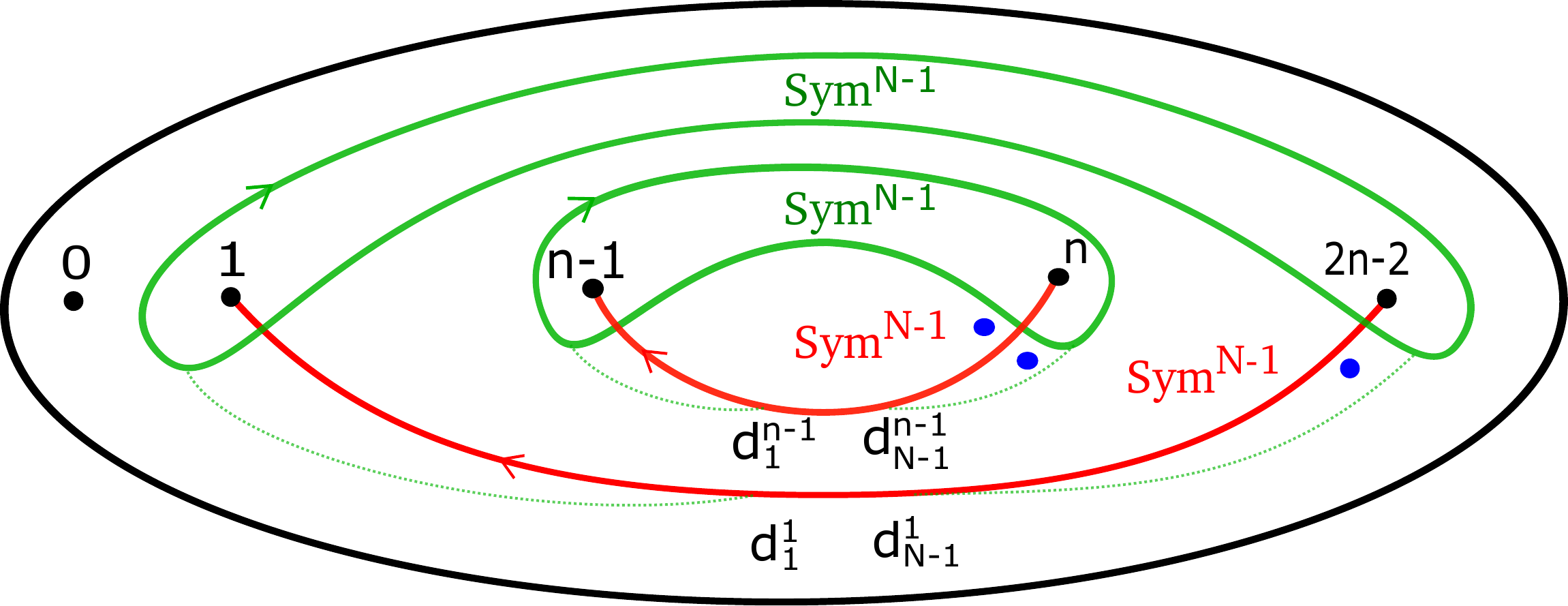}
\caption{Lagrangian submanifolds}
\label{Picture'}
\end{figure}
\vspace{-5mm}
Using that the braid group $B_{3n-1}$ is the mapping class group of the $(3n-1)-$punctured disc, we construct a well-defined Lagrangian $(\beta_n\cup \mathbb I_{2n-1}) \ \mathscr S_{n}^N \subseteq\Sigma^{n,N}$
associated to a braid $\beta_n\in B_n$. 

In definition \ref{Dint} we define an intersection pairing $\langle , \rangle$ taking values in $\Z[x^{\pm 1},y^{\pm 1},d^{\pm 1}]$, which will be computed from the {\em intersection points between such submanifolds} in the symmetric power, {\em graded by Laurent polynomials} using a collection of loops in $\Sigma^{n,N}$. 

For $c \in \Z$, we use a specialisation $\psi_{c,q,\lambda}: \Z[u^{\pm 1},x^{\pm1},y^{\pm1},d^{\pm1}]\rightarrow \Z[q^{\pm 1},q^{\pm \lambda}]$ given by:
\begin{equation}\label{D:1''} 
 \psi_{c,q,\lambda}(u)= q^{c \lambda}; \ \ \psi_{c,q,\lambda}(x)= q^{2 \lambda}; \ \ \psi_{c,q,\lambda}(y)=-q^{-2}; \ \ \psi_{c,q,\lambda}(d)=q^{-2}.
\end{equation}
Also, we denote $\xi_N=e^{\frac{2 \pi i}{2N}}$, $w(\beta_n)$ the writhe of $\beta_n$ and $\mathbb I_{2n-1}$ the trivial braid with $2n-1$ strands.
\begin{thm}[Unified model as an intersection of two Lagrangians]\label{THEOREM} Let us fix $N \in \N$ the colour of our invariants. Let $L$ be an oriented link and $\beta_n \in B_n$ such that $L=\hat{\beta}_n$. We consider the Lagrangian submanifolds $\cs, \ct  \subseteq \Sigma^{n,N}$
and define $\Gamma_N(\beta_n)(u,x,y,d) \in \Z[u^{\pm1},x^{\pm 1},y^{\pm 1},d^{\pm 1}]$ to be the following graded intersection:
\begin{equation}\label{eq:0'}
\Gamma_N(\beta_n)(u,x,y,d):=u^{-w(\beta_n)} u^{-(n-1)} (-y)^{-(n-1)(N-1)}\langle (\beta_n\cup \mathbb I_{2n-1}) \ \mathscr S_{n}^N,\mathscr T_{n}^N  \rangle.
\end{equation}
Then, $\Gamma_N$ recovers the $N^{th}$ coloured Jones and $N^{th}$ coloured Alexander invariants through the following specialisations of coefficients:
\begin{equation}
\begin{aligned}
&J_N(L,q)=\Gamma_N(\beta_n)|_{\psi_{1,q,N-1}}\\
&\Phi_{N}(L,\lambda)=\Gamma_N(\beta_n)|_{\psi_{1-N,\xi_N,\lambda}}.
\end{aligned}
\end{equation}
\end{thm}
\subsection{Construction of the graded intersection} Let ${\bf d}:=(d^1_1,...,d^1_{N-1},...,d^{n-1}_1,...,d^{n-1}_{N-1})$ be the base point from figure \ref{Picture'}. We consider $\mathbf L$ to be a set of loops in the symmetric power which are allowed to intersect the diagonal transversely and they do intersect it at most once on the image of a semicircle, as in definition \ref{lloo}. Then, in relation \eqref{G} we define a grading on this set of loops 
\vspace{-2mm}
\begin{equation}
\begin{aligned}
&\phi: \mathbf L \rightarrow \Z\oplus\Z\oplus\Z\\
&\hspace{13mm}\langle x \rangle \ \langle y \rangle \ \langle d \rangle.\\
\end{aligned}
\end{equation}
The components of $\phi(l)$ for $l \in \mathbf L$ are defined as below:
\begin{itemize}
\setlength\itemsep{-0.2em}
\item[•]$x$-coefficients count the winding of the loop $l$ around the $2n-1$ black punctures 
\item[•]$y$-coefficients record the winding of $l$ around the $n$ blue punctures
\item[•]$d$-coefficients count the relative winding of the loop $l$ in the symmetric power, give by the winding around the diagonal. 
\end{itemize}
The intersection pairing is indexed by the set of intersection points:
$I_{\Gamma_N(\beta)}=(\beta_n\cup \mathbb I_{2n-1}) \ \mathscr S_{n}^N \cap \mathscr T_{n}^N.$ We define a colouring $F=(f^1,...,f^{n-1})$ of such an intersection point $\bar{x}$ as a collection of functions $f^k:\{1,...,N-1\}\rightarrow\{1,...,N-1\}$ for all $k \in \{1,...,n-1\}$ such that the partitions of the set $\{1,...,N-1\}$ induced by the pre-images of points through each of these functions have cardinalities prescribed by the multiplicities of the components of $\bar{x}$ (definition \ref{loops}).
Let us denote by $\Co(\bar{x})$ the set of such colourings. Then each colouring $F$ prescribes a loop $l_{\bar{x},F}$ in the symmetric power, based at $\bf d$, as in definition \ref{defloop}. These loops are defined explicitly and use the dotted paths which link the red geometric support with the green geometric support.
\begin{defn}(Grading of an intersection point)
Let $\bar{x}\in I_{\Gamma_N(\beta)} $. We consider the polynomial associated to this intersection point given by the gradings of the family of loops associated to $\bar{x}$ and all possible colourings $F \in \Co(\bar{x})$:
\begin{equation}
P(\bar{x}):=\sum_{F \in \Co(\bar{x})} \phi(l_{\bar{x},F}) \in \Z[x^{\pm 1},y^{\pm 1}, d^{\pm 1}].
\end{equation}
Also, we denote by $\epsilon(\bar{x})$ the product of the signs of the local intersections between the supports of $(\beta_n\cup \mathbb I_{2n-1}) \ \mathscr S_{n}^N$ and $\mathscr T_{n}^N$ at the $(n-1)(N-1)$ components of $\bar{x}$. 
\end{defn}
\begin{defn}(Intersection form) \label{Dint} Then, the intersection pairing is given by the intersection points graded by the above polynomials:
\begin{equation}
\langle (\beta_n\cup \mathbb I_{2n-1}) \ \mathscr S_{n}^N,\mathscr T_{n}^N  \rangle=\sum_{\substack{{\bar{x} \in (\beta_n \cup \mathbb I_{2n-1})\mathscr S_{n}^N\cap\mathscr T_{n}^N }}} \epsilon(\bar{x}) \cdot P(\bar{x}).
\end{equation}
\end{defn}
\begin{defn}(Unified formula as a graded intersection) \label{MC}The graded intersection is given by:
\begin{equation}\label{eq:0}
\begin{aligned}
&\Gamma_N(\beta_n)(u,x,y,d):=u^{-w(\beta_n)} u^{-(n-1)} (-y)^{-(n-1)(N-1)}\sum_{\substack{{\bar{x} \in (\beta_n \cup \mathbb I_{2n-1})\mathscr S_{n}^N\cap\mathscr T_{n}^N }}} \epsilon(\bar{x}) \cdot P(\bar{x}).
\end{aligned}
\end{equation}
\end{defn}
\subsection{State sum models and the globalised model} The advantage of this model is that it uses the intersection points between two fixed Lagrangians, which, as we will see already for $N=2$, brings the model close to Heegaard diagrams and tools used for categorifications like Floer homologies. For higher colours, in the {\em previous state sum model from \cite{Cr2}} we would get many terms, and for each of them one has to compute the graded intersection of two Lagrangians. {\em Here, when we increase the colour} we just get more possible multiplicities for the intersection points, which has the effect of having more possible colourings. However, the {\em difference between two colourings of a fixed point} is directly seen as a { \em power of the variable $d$}, so it is easy to handle. 
\subsection{Role of the blue punctures} We remark that the grading in this formula counts the winding around the black punctures, blue punctures and a relative winding. The monodromy around the blue punctures encodes precisely the quantum trace from the definition of coloured Jones and Alexander polynomials from the representation theory. Also, as we will see below, the addition of blue punctures together with the relative winding permits us to pass from the Alexander polynomial towards the Jones polynomial. 
\subsection{Jones and Alexander polynomials} For the particular case when $N=2$, we have an intersection model for the Jones and Alexander invariants, in Section \ref{S:5'}. In this case, for each intersection point we have a unique colouring and our intersection model is obtained from the {\em intersection points between two Lagrangians in the configuration space}, graded by {\em monomials}.

\clearpage
$${\color{red}(\beta_n \cup \mathbb I_{2n-1})\mathscr S_{n}^2}\cap{\color{dgreen}\mathscr T_{n}^2}$$
\vspace{-11mm}
\begin{center}
\begin{figure}[H]
\centering
\includegraphics[scale=0.3]{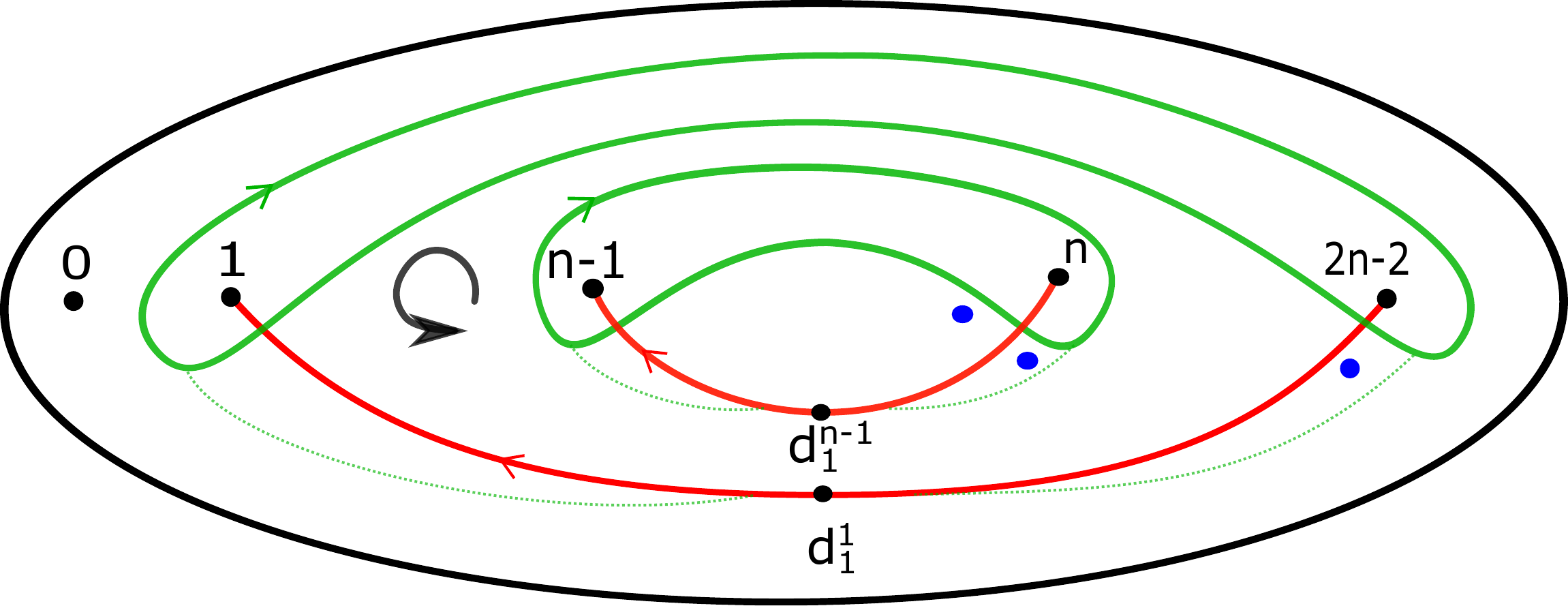}
\hspace{100mm}$\Large{Conf_{n-1}(\mathscr D_{3n-1})}$
\caption{Jones and Alexander polynomials}\label{F}
\end{figure}
\end{center}
\vspace{-13mm}
\begin{thm}(Jones and Alexander polynomials from the same graded intersection)\label{Theorem} The Jones and Alexander polynomials come from the intersection points between the two Lagrangians in the configuration space of the punctured disc, graded using the monodromy around punctures and the relative winding in the configuration space:
\begin{equation}\label{eqC:3}
\begin{aligned}
&J(L,q)= \Gamma_2(\beta_n)|_{u=q;x=q^{2},y=-q^{-2}, d=q^{-2}}\\
&\Delta(L,t)=\Gamma_2(\beta_n) |_{u=t^{-\frac{1}{2}};x=t;y=1;d=-1}. \end{aligned}
\end{equation}
\end{thm}
\subsection{Differences between the two specialisations for Jones and Alexander cases} Further on, in Subsection \ref{SS} we study this model for the case of the Alexander polynomial. We show that in this case the {\em blue punctures play no role}, and the grading which encodes the relative winding, given by the intersection with the diagonal of the symmetric power can be removed as well. In Theorem \ref{THEOREM'''} we show that the {\em Alexander polynomial} of the closure of $\beta_n$ is an intersection between {\em two Lagrangians in the configuration space of $(n-1)$ points} in the $(2n-1)$-punctured disc. The {\em grading takes care of just the winding around the punctures} of this punctured disc.  

Moreover, we remark that the picture that is used for this model, figure \ref{FF}, is actually a multi-pointed Heegaard diagram, if we consider the genus $(n-1)$ surface obtained by gluing $(n-1)$ handles at the ends of each red curve. Also, the intersection model of Corollary \ref{THEOREM'''} is very close to the one encountered in knot Floer homology. The advantage of the unified model of Theorem \ref{Theorem} is that it gives the Jones polynomial through the same tools, if one takes care of an {\em extra grading coming from the intersections with the diagonal of the symmetric power}.

\subsection{Further directions- Categorifications} 
For the case $N=2$ we expect that there exists a definition of a Floer type construction starting from the model $\Gamma_2(\beta_n)|_{\psi_{-1,i,\lambda}}$ which gives knot Floer homology. Our Lagrangians are very similar to the pictures encountered in this theory and the gradings are very closely related. Our next part of the program is to construct a categorification for the Jones polynomial starting from $\Gamma_2(\beta_n)|_{\psi_{1,q,-1}}$. There are examples suggesting that this is related to Khovanov homology and the fact that we have both theories from the same picture would create a proper framework to define a spectral sequence from Khovanov homology towards knot Floer homology.

Further on, the form of the Lagrangians gives a natural framework to study categorifications also for higher colours, for the families of coloured Jones and coloured Alexander polynomials. This improves on and is different from the previous results from \cite{Cr2}, where we had sums of Lagrangian intersections, which were not yet globalised in a geometric manner.

{\bf Asymptotics} This result brings together the $U_q(sl(2))$ invariants of colour $N$ as an intersection between two Lagrangians in a symmetric power of order $(n-1)(N-1)$. The colour can be seen in the number of particles of this symmetric power and formula \ref{eq:0'} gives the possibility to investigate geometrically the behaviour of this form when we increase the colour. This provides a starting point for investigating geometric models for asymptotics given by Melvin-Morton-Rozansky's loop expansions (\cite{MM}) and Gukov-Manolescu (\cite{GM}).
\subsection*{Structure of the paper} In Section \ref{S:1} we define the Lagrangian submanifolds, and we construct the intersection pairing in the symmetric power. In the following part, Section \ref{S:2}, we recall the state sum model from \cite{Cr2}. Then, in Section \ref{S:3}, we show that we can read the state sum model from our grading in the symmetric power. In Section \ref{S:4} we prove the intersection model. Section \ref{S:5'} concerns the intersection model for the case given by Jones and Alexander polynomials. The last part, Section \ref{S:5}, shows two concrete examples of computations for the Alexander and Jones polynomials of the trefoil knot and the knot $8_{19}$.

\subsection{Acknowledgments} I would like to thank Andr\'as Juh\'asz for discussions concerning the case $N=2$ and Jacob Rasmussen for asking me this main question of globalising the classes that appeared in the previous  state sum. I acknowledge the support of SwissMAP, a National Centre of Competence in Research funded by the Swiss National Science Foundation.
\section{Construction of the two submanifolds and the definition of the grading}\label{S:1}
In this part, for $n,N \in \N$ we present the homological set-up and the construction of the two Lagrangians which will lead to the colour $N$ quantum invariants. 

We start with the two dimensional disc with boundary, with $n_1+n_2$ punctures:
$$\mathscr D_{n_1,n_2}=\mathbb D \setminus \{p_1,...,p_{n_1},q_1,...,q_{n_2}\}.$$
Further on, we consider the $m^{\text{th}}$ symmetric power of this punctured disc:
\begin{align*}
&\Sy_{n_1,n_2}^m:=\mathscr D_{n_1,n_2}^{\times m} /Sym_m\\
&\bar{\Sy}^m:=\mathbb D^{\times m}/Sym_m
\end{align*}
(here $Sym_m$ is the $m^{\text{th}}$ symmetric group).

Now, let us fix $d_1,..d_{m} \in \mathscr D_{n_1,n_2}$ and ${\bf d}=(d_1,...,d_{m})\in \Sy_{n_1,n_2}^m$ the associated base point in the symmetric power.

$$ \hspace{12mm }\Delta^{k}_P \ \ \ \ \ \ \ \ \ \ \ \ \Delta^{-k}_P$$
\vspace{-18mm}
\begin{figure}[H]
\centering
\includegraphics[scale=0.4]{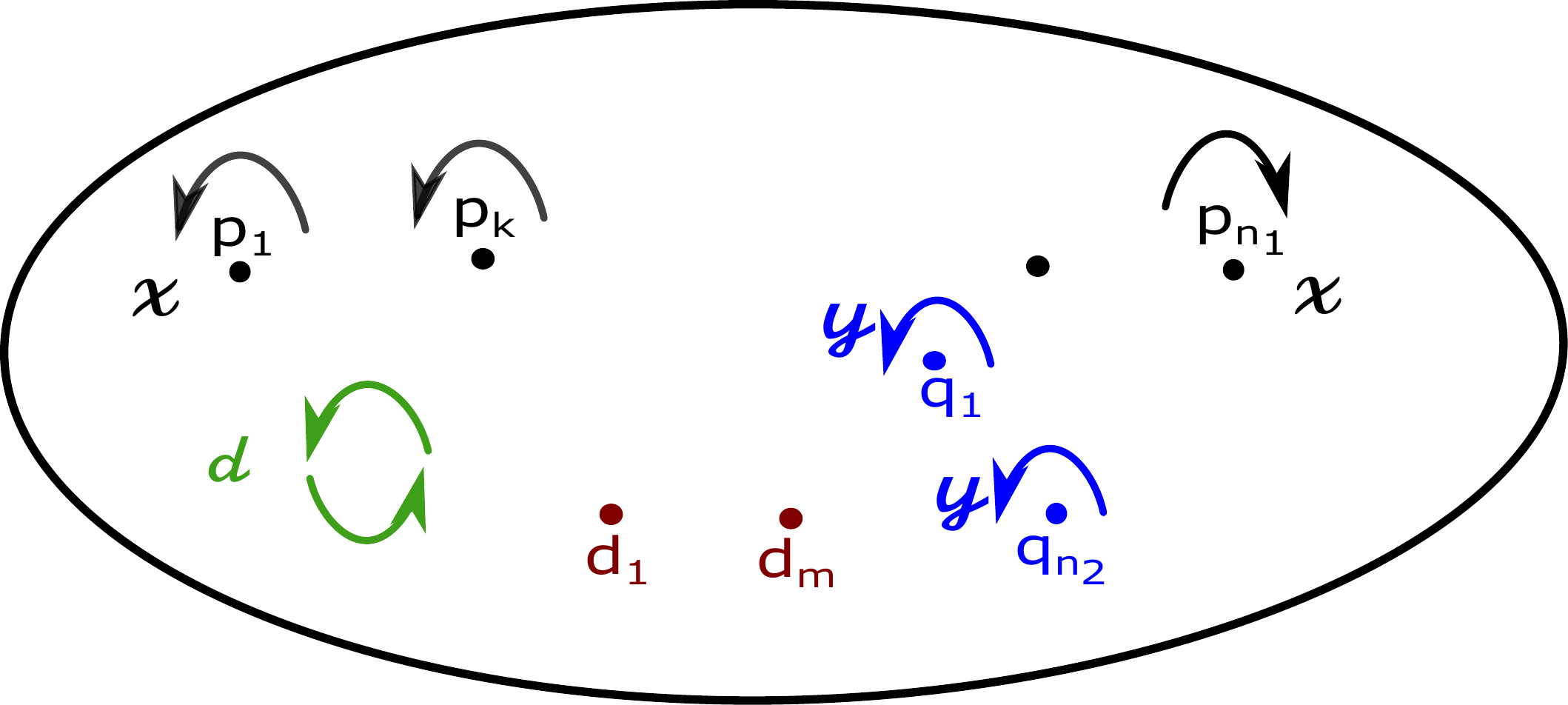}
\vspace{-20mm}
$$ \hspace{14mm }\Delta \hspace{40mm}\Delta_Q$$
\vspace{5mm}
\caption{Grading on the symmetric power of the punctured disc}
\label{Picture'''}
\end{figure}
\begin{notation}(Diagonals)
Let us denote the following submanifolds in the symmetric power:
\begin{align*}
&\Delta:=\lbrace (x_1,...,x_m) \in \Sy_{n_1,n_2}^m \mid \exists \ 1\leq i,j \leq m, x_i=x_j \rbrace\\
&\Delta^{k}_P:=\lbrace (x_1,...,x_m) \in \ \bar{\Sy}^m \mid \exists \ 1\leq i \leq m, 1\leq j \leq k, x_i=p_j \rbrace\\
&\Delta^{-k}_P:=\lbrace (x_1,...,x_m) \in \bar{\Sy}^m \mid \exists \ 1\leq i \leq m, n_1-k+1\leq j \leq n_1, x_i=p_j \rbrace\\
&\Delta_Q:=\lbrace (x_1,...,x_m) \in \bar{\Sy}^m \mid \exists \ 1\leq i \leq m, 1\leq j \leq n_2, x_i=q_j \rbrace
\end{align*}
for $k \in \{1,...,n_1\}$.
\end{notation}
In the following part, we consider a particular set of loops in the symmetric power, which will be used for the grading of our intersection pairing. For this, we denote by $S^1$ the radius $1$ circle in $\mathbb R^2$ and by $S^1_-$, $S^1_+$ the semicircles given by points from $S^1$ with a negative or non-negative $x-$coordinate respectively. 
\begin{defn}[Isotopy classes of loops] \label{lloo}We will work with the following classes of loops:
\begin{align*}
&\mathbf L:=\{l: S^1\rightarrow \Sy_{n_1,n_2}^m \text{ smooth }\mid l^{-1}(\Delta) \cap S^1_{+}\subseteq \{1\}, \ l(S^1_{-}) \text{ is transverse to } \Delta \}\\
&\Omega:=\pi_0\left(\mathbf L \right).
\end{align*}
\end{defn}
\subsection{Local system}\label{S:locsyst}
\begin{defn}[Grading on the space $L$]
Let $l \in \mathbf L$ be a loop and consider $\bar{l}:S^1\rightarrow \bar{\Sy}^m$ the associated loop in the symmetric power of the disc. Since $\pi_1(\bar{\Sy}^m)\simeq 1$ there exists $$\sigma_l:(\mathbb D^2,\partial \mathbb D^2)\rightarrow(\bar{\Sy}^m,\bar{l})$$ a disc which has as boundary the loop $\bar{l}$. Moreover, by a version of the Whitney Approximation Theorem, one can choose a smooth representative of such a disc. Furthermore, by the Thom Transversality Theorem, we choose a smooth representative of this disc which is transverse to the diagonals $\Delta$,$\Delta^{k}_P$,$\Delta^{-k}_P$ and $\Delta_Q$. 
We remark that $\Delta=\cup_{i,j=1}^{m}\Delta_{i,j}$ where $$\Delta_{i,j}:=\lbrace (x_1,...,x_m) \in \Sy_{n_1,n_2}^m \mid x_i=x_j \rbrace
.$$
Let us fix an orientation for all components $\Delta_{i,j}$, and also for the components of 
$\Delta^{k}_P$, $\Delta^{-k}_P$ and $\Delta_Q$.
We define the following grading on the space $\mathbf L$: 
\begin{equation}\label{G}
\begin{aligned}
&\phi: \mathbf L \rightarrow \Z\oplus\Z\oplus\Z\\
&\hspace{13mm}\langle x \rangle \ \langle y \rangle \ \langle d \rangle\\
&\phi(l):=x^{(\sigma_l,\Delta_P^k)-\left(\sigma_l,\Delta_P^{-(n_1-k)}\right)} \cdot y^{(\sigma_l ,\Delta_Q)} \cdot d^{(\sigma_l \cup~l(S^1_{-}),\Delta)}.
\end{aligned}
\end{equation}
The group $\Z\oplus\Z\oplus\Z$ should be thought of as the multiplicative group of monomials.
Picture \ref{Picture'''} shows a pictorial representation of this grading.
\end{defn}
Here $(\cdot,\cdot)$ is the geometric intersection number in $\bar{\Sy}^m$, which is defined as the sum of the intersection numbers with each of the components. The grading $\phi$ is well defined on the space $\mathbf L$ and it does not depend on the choice of a representative of the disc. So we have the following induced grading.
\begin{defn}[Grading on the set $\Omega$]\label{P:grading}
The function $\phi$ on $\mathbf L$ induces a well defined grading on its set of connected components, which we denote by:
\begin{align*}
&\varphi:\Omega\rightarrow \Z\oplus\Z\oplus\Z\\
&\hspace{13mm}\langle x \rangle \ \langle y \rangle \ \langle d \rangle\\
&\varphi([l]):=\phi(l).
\end{align*}
\end{defn}
\subsection{Braid group action}
In the following part we restrict to the context that we need and define a graded intersection between submanifolds of the symmetric power, which have the form that occurs in our situation. More precisely, we fix $n,N \in \N$ and we use the following parameters:
\begin{equation}
\begin{cases}
& n_1=2n-1 \ \ \ n_2=n\\
& k=n\\
& m=(n-1)(N-1).
\end{cases}
\end{equation}
We denote the symmetric power of the (punctured) disc as follows:
\begin{equation}
\begin{cases}
& \Sy^{n,N}:=\Sy_{2n-1,n-1}^{(n-1)(N-1)}\\
& \bar{\Sy}^{n,N}:=\bar{\Sy}^{(n-1)(N-1)}.
\end{cases}
\end{equation}
The intersection pairing will be a twisted version of the geometric intersection in $\Sy^{n,N}$, graded by the function $\varphi$. In order to prescribe the grading, we need a base point and we choose it on the submanifold $\cs$ and denote it by ${\bf d}:=(d^1_1,...,d^1_{N-1},...,d^{n-1}_1,...,d^{n-1}_{N-1})$, as in picture \ref{Diffeo}.
\begin{defn}We denote by
$s\cs, s\ct \subseteq \mathscr D_{2n-1,n}$ the collections of red curves and green circles that are the geometric supports of $\cs$ and $\ct$ respectively.
\end{defn}
\begin{defn}(Paths to the base points) Let us choose a collection of $(n-1)$ paths in the punctured disc, which we denote by $\eta_1,...,\eta_{(n-1)}$ which start from the red curves and end on the left hand side of the circles, as in figure \ref{Diffeo}. Similarly, let $\eta'_1,...,\eta'_{(n-1)}$ be the paths which start from the red curves and end on the right hand side of the circles.\end{defn}
We use that the braid group $B_{n}$ is the mapping class group of the punctured disc and so for $\beta_n \in B_n$ we get a diffeomorphism $h_{\beta_n \cup \mathbb I_{2n-2}}\in \Diff(\mathscr D_{2n-1,n})$. 
\begin{rmk}\label{cond}
Up to isotopy we can choose this diffeomorphism such that it is the identity outside the dark grey disc around the punctures labeled by $0,...,n-1$ from picture \ref{Diffeo}. Also, we choose a representative $h_{\beta_n \cup \mathbb I_{2n-2}}$ such that there are no trivial discs in $$\mathscr D_{2n-1,n}\setminus \left(h_{\beta_n\cup \mathbb I_{2n-2}}(s\cs) \cup s\ct\right)$$ in the interior of this dark grey disc, with half of their boundary on $(h_{\beta_n\cup \mathbb I_{2n-2}}) (s\cs)$ and the other half on $s\ct$. 
\end{rmk}
This map will induce a diffeomorphism on the symmetric power of the punctured disc and we act by this on the submanifold $\cs$. 
\subsection{Loops associated to intersection points}
We will define a graded intersection between submanifolds of the form $$h_{\beta_n\cup \mathbb I_{2n-2}}( \cs) \text{ and } \ct, \text{ for } \beta_n \in B_n $$  for $h_{\beta_n\cup \mathbb I_{2n-2}}$ that satisfies the above mentioned conditions. This pairing will be parametrised by the intersection points between $h_{\beta_n\cup \mathbb I_{2n-2}}( \cs)$ and $\ct$ and let us denote this set by $I_{h_{\beta_n}}$. Since the diffeomorphism $h_{\beta_n\cup \mathbb I_{2n-2}}$ was trivial outside the dark grey disc from the picture we have that:
$${\bf d}\in h_{\beta_n\cup \mathbb I_{2n-2}}( \cs).$$
\vspace{-7mm}
\begin{figure}[H]
\centering
\includegraphics[scale=0.4]{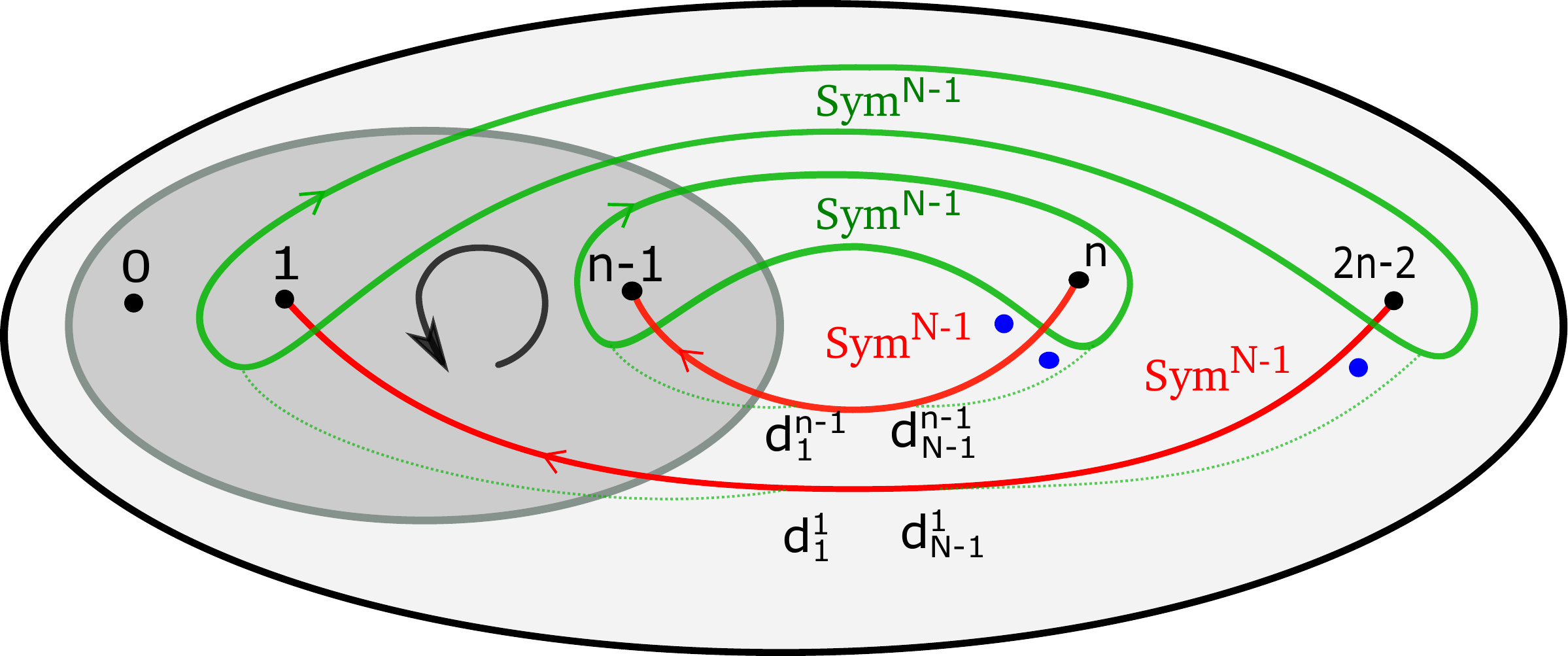}
\vspace{-1mm}
\caption{Braid action}
\label{Diffeo}
\end{figure}

\begin{defn}(Colourings associated to an intersection point) \label{loops}Let $\bar{x}=(x_1,...,x_{(n-1)(N-1)}) \in I_{h_{\beta_n}}$.
Let $k \in \{1,...,n-1\}$ and now we look at the $k^{th}$ red curve which ends in the puncture labeled by $2n-1-k$.
There are precisely $N-1$ components of $\bar{x}$ which belong to this curve (counted with multiplicities) and denote them and their multiplicities by:
\begin{align*}
&(x^k_1,...,x^k_{l_k})\\
&(m^k_1,...,m^k_{l_k}), \text{ where } m^k_1+...+m^k_{l_k}=N-1.
\end{align*}
(we denote them using the opposite orientation of the red segment: we start from the left hand side of the disc and go towards the right hand side)

In order to define a {\em colouring} of the point $\bar{x}$, we need a collection of functions $f^k:\{1,...,N-1\}\rightarrow\{1,...,l_k\}$ for all $k \in \{1,...,n-1\}$ with the property that: 
\begin{itemize}
\item[•]$\card \left((f^k)^{-1}\{i\}\right)=m^k_i$.\\
\item[•]$(f^k)^{-1}\{l_{k}\}=\Bigl\lbrace N-m^k_{l_k},N-(m^k_{l_k}-1),...,N-2,N-1\Bigr\rbrace$ if  $x^k_{l_k}$  belongs to the right hand side of the disc.
\end{itemize}
Let us denote:
\begin{equation}
\left(f^k\right)^{-1}\{i\}=\Bigl\lbrace f^k[i,1],...,f^k[i,m^k_i] \Bigr\rbrace.
\end{equation}
We consider the family of these functions $F:=(f^1,...,f^{n-1})$ and call this a {\em colouring of the multipoint $\bar{x}$}.
\end{defn}
\begin{defn}For a given intersection point $\bar{x}$, we denote the set of associated colourings by:
$$\Co(\bar{x}):=\lbrace F=(f^1,...,f^{n-1}) \mid F \text{ colouring with respect to the multiplicities of the multipoint } \bar{x} \rbrace.$$
\end{defn}
In the following part, for a given intersection point $\bar{x}$ and a colouring $F$ we will construct a loop $l_{\bar{x},F}$ in the symmetric power of the disc, based in $\bf{d}$. This will be done in two steps:
\begin{enumerate}
\item[•]First, we construct a path from $\bf d$ towards $\bar{x}$, which does not depend on the colouring and uses the green submanifold $\ct$.
\item[•]Secondly, we continue with a path from $\bar{x}$ back to the basepoint $\bf d$, using the colouring $F$ and the red submanifold $\cs$.
\end{enumerate}
Then, we will concatenate these paths and obtain a loop in the symmetric power. Let us make this precise, as follows.

\begin{defn}(Path from the base point towards an intersection point)\\ 
Let $\bar{x}=(x_1,...,x_{(n-1)(N-1)}) \in I_{h_{\beta_n}}$.
Let $k \in \{1,...,n-1\}$ and now we look at the $k^{th}$ green circle which goes around the punctures labeled by $(k, 2n-1-k)$.
There are precisely $N-1$ components of $\bar{x}$ which belong to this circle (counted with multiplicities) and denote them and their multiplicities by:
\begin{align*}
&(x^k_1,...,x^k_{l_k})\\
&(\bar{m}^k_1,...,\bar{m}^k_{l_k}), \text{ where } \bar{m}^k_1+...+\bar{m}^k_{l_k}=N-1.
\end{align*}
by ordering them as they occur on the circle minus the maximal point using the opposite of its orientation.
If $x^k_1$ is in the left half of the disc, let $\bar{\nu}^k_1,...,\bar{\nu}^k_{\bar{m}_1}$ be $\bar{m}^k_1$ paths in the punctured disc which start in $d^k_1,...,d^k_{\bar{m}^k_1}$, follow the path $\eta_k$ and then they follow the green curve until they reach the point $x^k_1$.
\begin{figure}[H]
\centering
\includegraphics[scale=0.4]{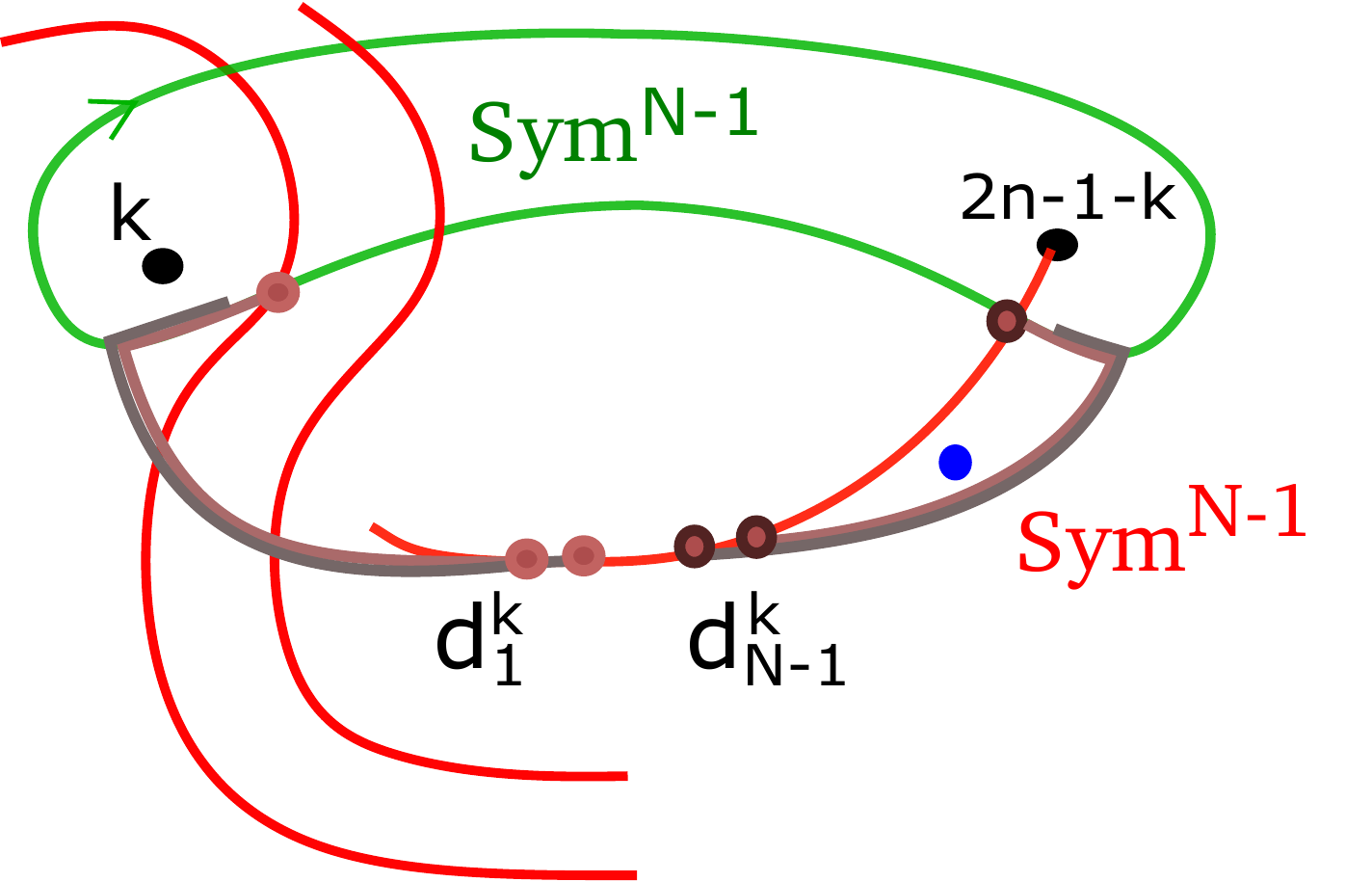}
\vspace{-1mm}
\caption{Paths from the base points}
\label{Picture}
\end{figure}
Now we look at the point $x^k_2$.  If this is in the left hand side of the disc as well, then we construct in a similar manner $\bar{m}^k_2$ paths starting from the base points $d^k_{\bar{m}^k_1+1},...,d^k_{\bar{m}^k_1+\bar{m}^k_2}$ which all pass through $x^k_2$. We continue this procedure up to the point where we have $\bar{m}^k_{l-1}$ paths passing through $x^k_{l-1}$.

Looking at the picture we notice that we have maximum one intersection point in the right hand side of the disc, namely $x^k_{l_k}$. If so, then the last set of paths will be constructed in a slightly different manner. Let $\bar{\nu}^k_{N-\bar{m}^k_{l_k}},...,\bar{\nu}^k_{N-1}$ be a family of paths starting from the last base points from the $k^{th}$ red curve: $d^k_{N-\bar{m}^k_{l_k}},...,d^k_{N-1}$, follow the path $\eta'_k$ towards the green circle and then continue on the green support up to the point $x^k_{l_k}$.

We do this procedure for all $k\in \{1,...,n-1\}$ and we obtain a collection of $(n-1)(N-1)$ paths in the punctured disc.
Let us consider
\begin{equation}
\bar{\gamma}_{\bar{x}}:= \{\bar{\nu}^1_{1},...,\bar{\nu}^1_{N-1},...,\bar{\nu}^{n-1}_{1},...,\bar{\nu}^{n-1}_{N-1}\} 
\end{equation}
the path in $\Sy^{n,N}$ given by the collection of these paths (in the disc), which will start from $\bf d$ and end in $\bar{x}$.
\end{defn}
\begin{defn}(Path associated to a coloured intersection point, ending in the base point)\\
 Let $\bar{x}=(x_1,...,x_{(n-1)(N-1)}) \in I_{h_{\beta_n}}$ and suppose that $F$ is a colouring of $\bar{x}$.
Let $k \in \{1,...,n-1\}$ and this time we use the $k^{th}$ red curve which ends in the puncture labeled by $2n-1-k$.
On this curve, there are exactly $N-1$ components of $\bar{x}$ (counted with multiplicities) and we denote them as below:
\begin{align*}
&(x^k_1,...,x^k_{l_k})\\
&(m^k_1,...,m^k_{l_k}), \text{ where } m^k_1+...+m^k_{l_k}=N-1.
\end{align*}

The colouring will induce a partition of the base points $d^k_1,...,d^k_{N-1}$ into $l_k$ sets, which we denote by:
\begin{equation}
\begin{cases}
\lbrace d^k_{f^k[1,1]},...,d^k_{f^k[1,m^k_1]} \rbrace \\
\lbrace d^k_{f^k[l_k,1]},...,d^k_{f^k[l_k,m^k_{l_k}]} \rbrace.
\end{cases}
\end{equation}
To simplify the notation, we denote this partition of $N-1$ base points into $l_k$ sets of cardinals $m^k_1,...,m^k_{l_k}$ as follows:
\begin{equation}
\begin{cases}
\lbrace d^k_{[1,1]},...,d^k_{[1,m^k_1]} \rbrace \\
\lbrace d^k_{[l^k,1]},...,d^k_{[l^k,m^k_{l_k}]} \rbrace.
\end{cases}
\end{equation}

Let $\nu^k_1,...,\nu^k_{m^k_1}$ be $m^k_1$ paths in the punctured disc which start in $x^k_1$ and follow the red curve back to the base points,  arriving in $d^k_{[1,1]},...,d^k_{[1,m^k_1]}$ respectively. 
\begin{figure}[H]
\centering
\includegraphics[scale=0.4]{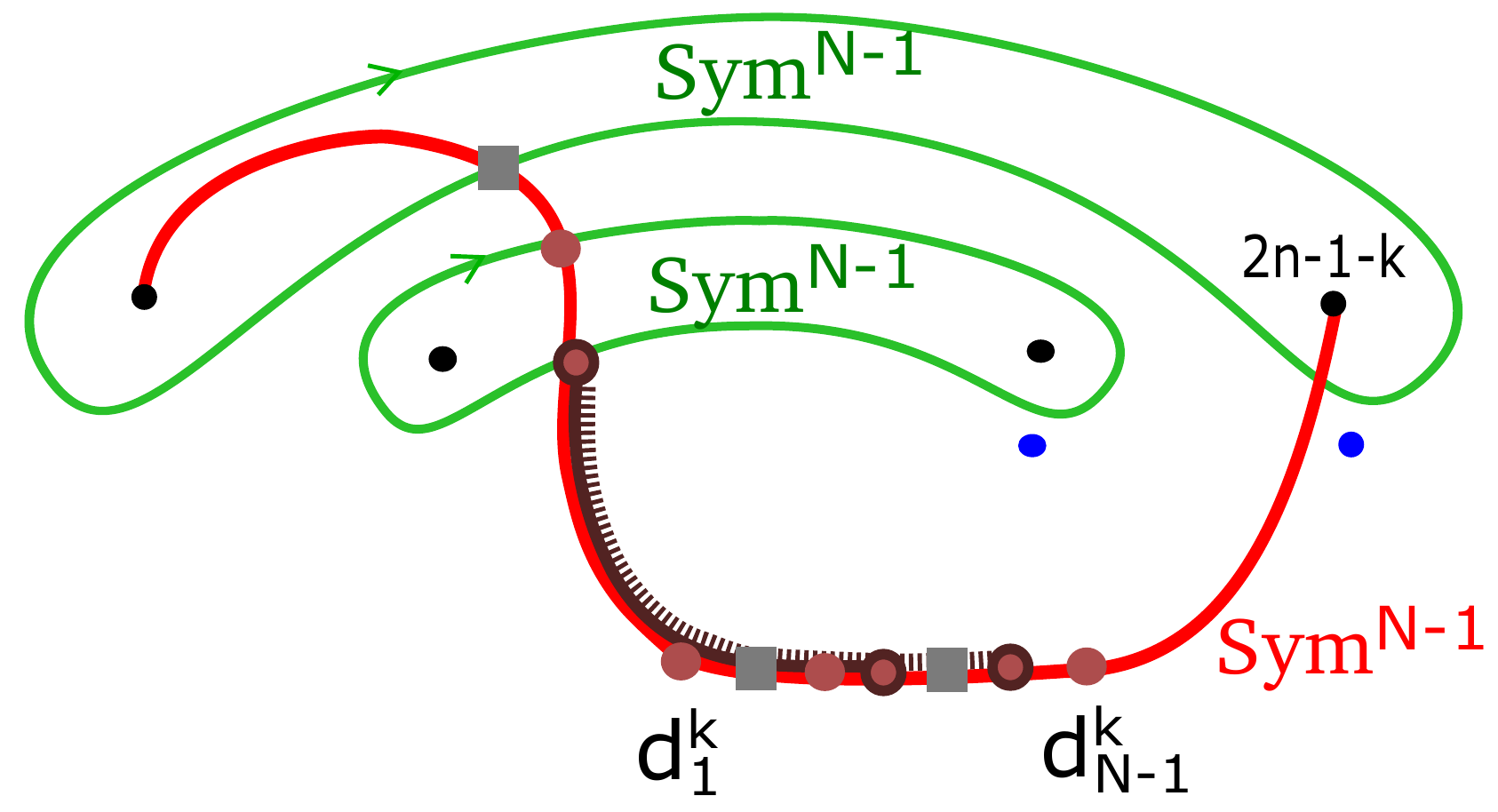}
\vspace{-1mm}
\caption{Paths to the base points}
\label{Picture}
\end{figure}
Now we look at the point $x^k_2$. We continue with an analogous construction, by defining $m^k_2$ paths starting from $x^k_2$ and going to the base points $d^k_{[2,1]},...,d^k_{[2,m^k_2]}$.
We continue this procedure up to the point where we have $m^k_{l_k}$ paths passing through $x^k_{l^k}$.
We do this procedure for all $k\in \{1,...,n-1\}$ and we obtain a collection of $(n-1)(N-1)$ paths in the punctured disc.
Let us consider
\begin{equation}
\gamma_{\bar{x},F}:= \{\eta^1_{1},...,\eta^1_{N-1},...,\eta^{n-1}_{1},...,\eta^{n-1}_{N-1}\} 
\end{equation}
the path in $\Sy^{n,N}$ which is given by the above set of paths, starting in $\bar{x}$ and ending in $\bf d$.
\end{defn}
Now, we put together the previous definitions and we define a loop in the symmetric power.
\begin{defn}(Loop associated to a coloured intersection point)\label{defloop}\\  
 Let $\bar{x}=(x_1,...,x_{(n-1)(N-1)}) \in I_{h_{\beta_n}}$ be an intersection point and consider $F$ to be a colouring of $\bar{x}$. We define the following loop based in $\bf d$, which passes through $\bar{x}$ and is determined by the colouring $F$:
 \begin{equation}
 l_{\bar{x},F}:=\gamma_{\bar{x},F} \circ \bar{\gamma}_{\bar{x}} \in \mathbf L.
 \end{equation}

\end{defn}
\subsection{Graded intersection}
\begin{defn}\label{D:int}We use the set of loops which pass through our intersection points in order to define a graded intersection 
$\langle (\beta_n\cup \mathbb I_{2n-1}) \ \mathscr S_{n}^N,\mathscr T_{n}^N  \rangle\in \Z[x^{\pm1},y^{\pm 1},d^{\pm 1}]$ as follows:
\begin{equation}\label{int}
 \langle h_{\beta_n \cup \mathbb I_{2n-2}}(\mathscr S_{n}^N),\mathscr T_{n}^N  \rangle:= \sum_{\bar{x}\in I_{h_{\beta_n}}} \epsilon_{x_1}\cdot...\cdot \epsilon_{x_{(n-1)(N-1)}}\cdot \left( \sum_{F \in \Co(\bar{x})} \varphi(l_{\bar{x},F}) \right).
\end{equation}
Here, we denote by $\epsilon_{x_i}$ the sign of the local intersection in the punctured disc at the point $x_i$ between the circle and the red curve that $x_i$ belong to.
\end{defn}
In the following part we show that the intersection form does not depend on the choice of representative $h_{\beta_n \cup \mathbb I_{2n-2}}$.
\begin{prop}(Intersections given by different representatives)\\
Let $\beta_n \in B_n$ a braid and consider two associated diffeomorphisms $h_{\beta_n \cup \mathbb I_{2n-2}}$ and $h'_{\beta_n \cup \mathbb I_{2n-2}}$ which satisfy the assertions from the previous subsection. Then 
\begin{equation}
 \langle h_{\beta_n \cup \mathbb I_{2n-2}}(\mathscr S_{n}^N),\mathscr T_{n}^N  \rangle= \langle h'_{\beta_n \cup \mathbb I_{2n-2}}(\mathscr S_{n}^N),\mathscr T_{n}^N  \rangle.
 \end{equation}
\end{prop}
\begin{proof}
The two intersection pairings are parametrised by the set of intersection points:
\begin{equation}\label{hhh}
\begin{cases}
I_{h_{\beta_n}}=h_{\beta_n \cup \mathbb I_{2n-2}}(\mathscr S_{n}^N)\cap\mathscr T_{n}^N\\
I_{h'_{\beta_n}}=h'_{\beta_n \cup \mathbb I_{2n-2}}(\mathscr S_{n}^N)\cap\mathscr T_{n}^N.
\end{cases}
\end{equation}
Now, following the condition that we do not have trivial discs (bigons) between the red curves and the green ones inside the grey disc neither in the picture corresponding to $h_{\beta_n}$ nor in the one for $h'_{\beta_n}$ we can find an isotopy between these two diffeomorphisms such that at each level there are no bigons in the disc.

It means that the isotopy between these two diffeomorphisms gives a one to one correspondence between the sets of intersection points:
$$f: I_{h_{\beta_n}}\rightarrow I_{h'_{\beta_n}}.$$
Now, let us look at an intersection point $\bar{x}\in I_{h_{\beta_n}}$ and its correspondent $f(\bar{x})\in I_{h'_{\beta_n}}$. We notice that the multiplicities of the point $\bar{x}$ are the same as the multiplicities of the point $f(\bar{x})$. We are interested in the gradings which are associated to these two intersection points. 

Let us fix a colouring $F \in \Co(\bar{x})=\Co(f(\bar{x}))$.
We consider the loops associated to the points $\bar{x}$ and $f(\bar{x})$, coloured with $F$:
$$ l_{\bar{x},F}, l_{f(\bar{x}),F} \in \mathbf L.$$
Using that $h_{\beta_n}$ and $h'_{\beta_n}$ are isotopic and the construction of $l_{\bar{x},F}$ and $l_{f(\bar{x}),F}$ from definition \ref{loops}, we conclude that these loops are also isotopic in $\mathbf L$ and so they are in the same connected component of $L$:
$$[l_{\bar{x},F}]=[l_{f(\bar{x},F})] \in \Omega.
$$
Following the definition of the grading we have that 
\begin{equation}
\varphi(l_{\bar{x},F})=\varphi\left(l_{f(\bar{x},F)}\right), \forall x\in I_{h_{\beta_n}}. 
\end{equation}
Using this relation and the definition of the graded intersection from equation  \eqref{int} we conclude the equality \eqref{hhh} . 
\end{proof}
The last discussion shows that our intersection form is well defined and it does not depend on the diffeomorphism that is chosen to represent the braid. For the next part we want to use the symplectic structure. We take the symplectic structure on the symmetric power $\Sigma^{N,n}$ which comes from the restriction of the symplectic structure of the symmetric power of the complex plane. Then, we represent $s \cs$ and $s\ct$ by collections of smooth curves, which are Lagrangians in the punctured disc. 
Thus $\cs$ and $\ct$, which are given by the symmetric powers of these Lagrangians in the punctured disc, will be Lagrangians in $\Sigma^{N,n}$.
\begin{prop}
For $\beta_n \in B_n$ there is a representative $h_{\beta_n\cup \mathbb I_{2n-2}}$ which is a symplectomorphism of $\Sigma^{n,N}$ such that the two conditions from remark \ref{cond} are satisfied.  
\end{prop}
\begin{defn}(Graded intersection) For $\beta_n \in B_n$, let $h_{\beta_n\cup \mathbb I_{2n-2}}$ be a symplectomorphism of $\Sigma^{n,N}$ as above. Following the above discussion, we have a well defined pairing as below:
\begin{align}\label{inter}
 \langle {(\beta_n \cup \mathbb I_{2n-2}})\mathscr S_{n}^N,\mathscr T_{n}^N  \rangle:=&
 \langle (h_{\beta_n \cup \mathbb I_{2n-2}}(\mathscr S_{n}^N),\mathscr T_{n}^N  \rangle=\\
=&\sum_{\bar{x}\in I_{h_{\beta_n}}} \epsilon_{x_1}\cdot...\cdot \epsilon_{x_{(n-1)(N-1)}}\cdot \left( \sum_{F \in \Co(\bar{x})} \varphi(l_{\bar{x},F}) \right).
\end{align}
\end{defn}
\section{State sum model}\label{S:2} In this part, we present the setting from \cite{Cr2}  where we constructed a state sum model recovering the coloured Jones and coloured Alexander polynomials. There we used graded intersections in configuration spaces rather than symmetric powers. 

We consider the unordered configuration space of $(n-1)(N-1)$ particles in the punctured disc $\mathcal D_{2n-1}$.  
\subsection{Homology classes}\label{homclasses}
For this model we have used the family of Lawrence representations which are homological representations of the braid group $B_n$  
on the homology of a $\mathbb Z \oplus \mathbb Z$-covering of the configuration space of $m$ particles in the $n$-punctured disc $\mathscr D_{n}$ (which we denote by $C_{n,m}$). More precisely, we have:
\begin{itemize}
\item[1)] Lawrence representations $H_{n,m}, H^{\partial}_{n,m}$ which are $\mathbb Z[x^{\pm 1},d^{\pm 1}]$-modules that carry a $B_n$-action.
\item[2)] Intersection pairing $ \ll, \gg : H_{n,m} \otimes H^{\partial}_{n,m}\rightarrow \mathbb Z[x^{\pm 1},d^{\pm1}]$ (Poincar\'e-Lefschetz type duality).
\end{itemize}
The precise definition of these homology groups is presented in \cite{Cr2} Section 3.

Let us fix $n, N\in \N$. In the following we will use the space $C_{2n-1,(n-1)(N-1)}$ and let ${\bf d}=\{d_1,...,d_{(n-1)(N-1)}\}$ a base point as in picture \ref{Statesum'}.
The construction of the homology classes in the covering of the configuration space will be done by drawing a set of disjoint curves in the punctured disc, considering products of ordered configuration spaces on those and then taking their quotient to the unordered configuration space. Using this, for any multi-index $\bar{i}=(i_1,...,i_{n-1}), i_k \in \{0,...,N-1\}, k\in \{1,...,n-1\}$ we define two Lagrangians denoted by: 
\begin{center}
${\color{red} F'_{\bar{i}} \subseteq C_{2n-1,(n-1)(N-1)}} \ \ \text{ and }\ \  {\color{dgreen} L'_{\bar{i}}\subseteq C_{2n-1,(n-1)(N-1)}}.$
\begin{figure}[H]
\centering
\includegraphics[scale=0.4]{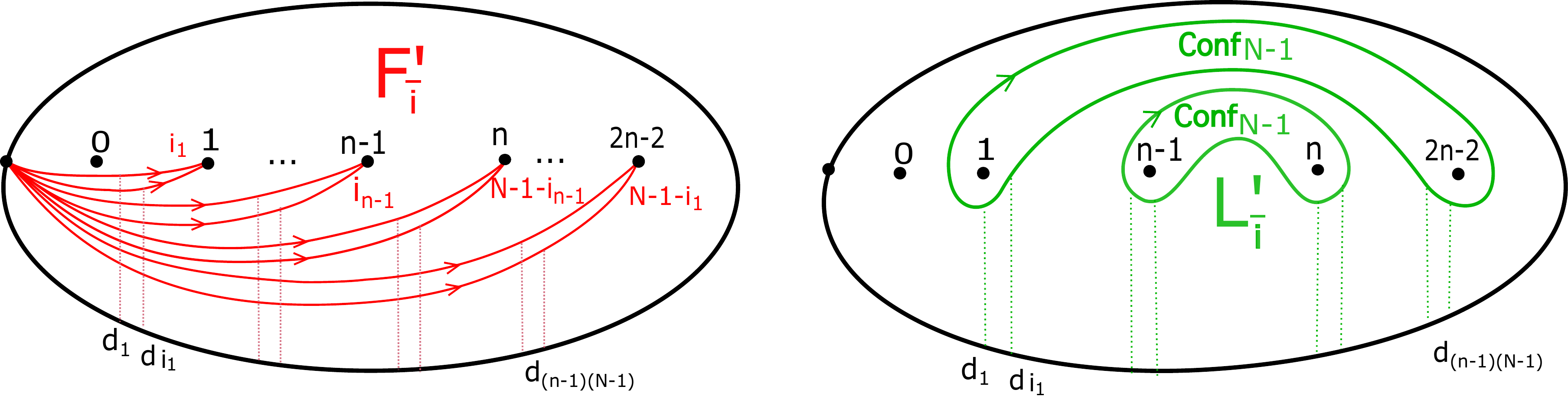}
\caption{State sum model}\label{Statesum'}
\vspace{-25mm}
$$\hspace{15mm} {\color{red}\eta^{F_{\bar{i}}}} \hspace{53mm} {\color{dgreen}\eta^{L_{\bar{i}}} }\hspace{14mm}$$
\vspace{2mm}
\end{figure}
\end{center}
In order to lift these submanifolds to the covering of this configuration space, we use two paths to the base point, which we denote by $\eta^{F_{\bar{i}}}$ and $\eta^{L_{\bar{i}}}$. Then, we consider
$${\color{red} \mathscr F'_{\bar{i}} \in H_{2n-1,(n-1)(N-1)}} \ \ \text{ and }\ \  {\color{dgreen} \mathscr L'_{\bar{i}}\in H^{\partial}_{2n-1,(n-1)(N-1)}}$$
to be the homology classes given by the lifts of these submanifolds, using the lifts of the paths to the base points. All this construction is presented in Definition 5.0.2, Definition 5.0.3 and Definition 6.2.1 from \cite{Cr2}. 
\begin{defn}[Specialisations]\label{D:1'''}   Let $c \in \Z$ and consider  the morphism:
\begin{equation}
\begin{aligned}
&\gamma_{c,q,\lambda}: \Z[u^{\pm 1},x^{\pm1},d^{\pm1}]\rightarrow \Z[q^{\pm 1},q^{\pm \lambda}]\\
& \gamma_{c,q,\lambda}(u)= q^{c \lambda}; \ \ \gamma_{c,q,\lambda}(x)= q^{2 \lambda}; \ \ \gamma_{c,q,\lambda}(d)=q^{-2}.
\end{aligned}
\end{equation}
\end{defn}

Then, we have the following model from \cite{Cr2}, Corollary 8.1.1.
\begin{thm}[Unified embedded state sum model \cite{Cr2}]\label{Thstate'}

Let $L$ be an oriented link and $\beta_n \in B_n$ such that $L=\hat{\beta}_n$.  Let us consider the polynomial in $3$ variables given by the following state sum:
\begin{equation}\label{Fstate'}
\begin{aligned}
\Lambda'_N(\beta_n)(u,x,d)&:=u^{-w(\beta_n)} u^{-(n-1)} \sum_{i_1,...,i_{n-1}=0}^{N-1} d^{-\sum_{k=1}^{n-1}i_k} \\
&\ll (\beta_{n} \cup {\mathbb I}_{n-1} ){ \mathscr F'_{\bar{i}}}, {\mathscr L'_{\bar{i}}}\gg \ \in \Z[u^{\pm1},x^{\pm 1},d^{\pm 1}].
\end{aligned}
\end{equation}
Then $\Lambda'_N$ gives the $N^{th}$ coloured Jones and $N^{th}$ coloured Alexander polynomials:
\begin{equation}
\begin{aligned}
&J_N(L,q)=\Lambda_N'(\beta_n)|_{\gamma_{1,q,N-1}}\\
&\Phi_{N}(L,\lambda)=\Lambda'_N(\beta_n)|_{\gamma_{1-N,\xi_N,\lambda}}.
\end{aligned}
\end{equation}
\end{thm}

\section{Computation of the intersection pairing from the state sum formula in the symmetric power setting}
\label{S:3}
The intersection pairing 
\begin{equation}\label{P:pairingconf}
\ll, \gg: H_{2n-1,(n-1)(N-1)} \otimes H^{\partial}_{2n-1,(n-1)(N-1)}\rightarrow \mathbb Z[x^{\pm 1},d^{\pm1}]
\end{equation}
is defined on the homology of the covering, but it can actually be computed from the geometric intersection in the base configuration space, graded by a certain local system, as in Proposition 3.3.2 from \cite{Cr2}. In the following part, we show that we can use the grading of loops from section \ref{S:locsyst} in order to compute this pairing.
 
 We start by presenting the computation from \cite{Cr2}. We fix the parameters from section \ref{S:locsyst} as below:
$$n_1\rightarrow2n-1, \ \ n_2\rightarrow0, \ \  m\rightarrow (n-1)(N-1).$$
We recall that the base configuration space that we work with is  $C_{2n-1,(n-1)(N-1)}$.
Now, for each intersection point $$\bar{w}=(w_1,...,w_{(n-1)(N-1)}) \in (\beta_{n} \cup {\mathbb I}_{n-1} ){ F'_{\bar{i}}}  \cap L'_{\bar{i}}$$ we construct a loop, denoted by $l_{\bar{w}} \subseteq C_{n,m}$. 
 We start with the path $\eta^{L_{\bar{i}}}$ which goes from $\bf d$ and it ends on $L'_{\bar{i}}$. We continue this path, following the geometric support given by the green circles (excluding the maximal point on each of them) in the punctured disc and construct a path in the configuration space from $\eta^{L_{\bar{i}}}(1)$ towards the intersection point $\bar{w}$, such that the image of this path is included in $L'_{\bar{i}}$. Basically, on each circle we choose a collection of $N-1$ paths in the punctured disc which go towards the $N-1$ components of $\bar{w}$ that lie on that circle, without passing through the maximal point of that circle. 
We denote the new path obtained by following the two previous paths by $\eta^{L_{\bar{i}}}_{\bar{w}}$, which goes from $\bf d$ to $\bar{w}$.

In the next step we aim to go back to the base point. We start from $\bar{w}$ and go towards the point $\eta^{F_{\bar{i}}}(1)$ by a path in the configuration space included in $(\beta_{n} \cup {\mathbb I}_{n-1} ){ F'_{\bar{i}}}$. In order to do this, we look at each red curve from the support of $(\beta_{n} \cup {\mathbb I}_{n-1} ){ F_{\bar{i}}}$ in the punctured disc, and consider a family of $N-1$ paths which start from the $N-1$ components of $\bar{w}$ which lie on this red curve and go towards $\eta^{F_{\bar{i}}}(1)$. We continue this path in the configuration space with $(\eta^{F_{\bar{i}}})^{-1}$ and go back to $\bf d$. We denote the path obtained by gluing these two paths by $\eta^{F_{\bar{i}}}_{\bar{w}}$, which goes from $\bar{w}$ back to $\bf d$.
 
Then, we define the loop which based in $\bf d$ and it is associated to the intersection point $\bar{w}$:
$$l_{\bar{w}}=\left(\eta^{F_{\bar{i}}}_{\bar{w}}\right)^{-1}\circ \eta^{L_{\bar{i}}}_{\bar{w}}.$$
This is a loop in the configuration space $C_{2n-1,(n-1)(N-1)}$. Further on,  we use the canonical inclusion and see this loop in the symmetric power $\Sy_{2n-1,0}^{(n-1)(N-1)}$. This means that we can use the grading defined in proposition \ref{P:grading} in order to evaluate the loop using the function $\varphi$. This leads to the intersection form as follows.
\begin{prop}[Pairing between homology classes computed via the symmetric power] \label{P:3}

\

The intersection pairing from \eqref{P:pairingconf} has the following formula (which uses the loops $l_{\bar{w}}$ and the grading):
\begin{equation}\label{eq:111}  
 \ll (\beta_{n} \cup {\mathbb I}_{n-1} ){ \mathscr F'_{\bar{i}}}, { \mathscr L'_{\bar{i}}}\gg ~= \sum_{\bar{w}\in (\beta_{n} \cup {\mathbb I}_{n-1} ){ F'_{\bar{i}}}  \cap L'_{\bar{i}}} \epsilon_{y_1}\cdot...\cdot \epsilon_{y_{(n-1)(N-1)}}\cdot \varphi(l_{\bar{w}})
 \in \Z[x^{\pm1}, d^{\pm1}].
\end{equation}
Here, for the sum from the right hand side of the equation we look at $F'_{\bar{i}}$ and $ L'_{\bar{i}}$ in the symmetric power $\Sy_{2n-1,0}^{(n-1)(N-1)}$ rather than in the configuration space. Also, $\epsilon_{w_1},...,\epsilon_{w_n}$ are signs of local intersections in the punctured disc, which are defined as the ones from definition \ref{D:int}.
\end{prop}
\begin{proof}
Since we are working in the symmetric power of the punctured disc in the situation where we have no blue punctures ($n_2=0$), it means that we have no $y$-variables in the grading $\varphi$. Moreover we have that 
\begin{equation}\label{eq:1}
\varphi(l_{\bar{w}})=\phi^{-(n-1)}(l_{\bar{w}})
\end{equation}
 where $\phi^{-(n-1)}$ is the local system from \cite{Cr2} presented in Definition 3.1.2. Then, using the formula from the pairing $\ll, \gg$ from Remark 3.3.3 (\cite{Cr2}) we have:
$$\ll (\beta_{n} \cup {\mathbb I}_{n-1} ){ \mathscr F'_{\bar{i}}}, {\mathscr L'_{\bar{i}}}\gg ~=\sum_{\bar{w}\in (\beta_{n} \cup {\mathbb I}_{n-1} ){ F'_{\bar{i}}}  \cap L'_{\bar{i}}} \epsilon_{w_1}\cdot...\cdot \epsilon_{w_{(n-1)(N-1)}}\cdot \phi^{-(n-1)}(l_{\bar{w}}) \in \Z[x^{\pm1}, d^{\pm1}].$$
This definition combined with the property from equation \eqref{eq:1} concludes the formula from the statement. 
\end{proof}

\section{Intersection model given by two Lagrangians}\label{S:4}

In this section we prove the intersection model from Theorem \ref{THEOREM}. We remind the statement below.
\begin{thm}[Unified model as an intersection of two Lagrangians]\label{THEOREM}
Let $L$ be an oriented link and $\beta_n \in B_n$ such that $L=\hat{\beta}_n$. We consider the following graded intersection:
\begin{equation}\label{eq:0}
\Gamma_N(\beta_n)(u,x,y,d)=u^{-w(\beta_n)} u^{-(n-1)}(-y)^{-(n-1)(N-1)} \langle (\beta_n\cup \mathbb I_{2n-1}) \ \mathscr S_{n}^N,\mathscr T_{n}^N  \rangle\in \Z[u^{\pm1},x^{\pm 1},y^{\pm 1},d^{\pm 1}].
\end{equation}
Then, $\Gamma_N$ recovers the $N^{th}$ coloured Jones and $N^{th}$ coloured Alexander invariants:
\begin{equation}
\begin{aligned}
&J_N(L,q)=\Gamma_N(\beta_n)|_{\psi_{1,q,N-1}}\\
&\Phi_{N}(L,\lambda)=\Gamma_N(\beta_n)|_{\psi_{1-N,\xi_N,\lambda}}.
\end{aligned}
\end{equation}
\end{thm}
We split the proof in four main steps. First, looking at the two types of Lagrangians that we introduced so far, in figure \ref{Picture'} and figure \ref{Statesum'}, we notice that the support of $\mathscr F'_{\bar{i}}$ and the one of $\mathscr S_n^N$ are very different. In the first step, we will consider a Lagrangian which is given by the geometric support of $\mathscr F'_{\bar{i}}$ by pushing the point from the boundary towards the middle of the disc, giving a homology class which we denote by $\mathscr F_{\bar{i}}$. We show that we have a state sum model using this family of homology classes. This model is useful on its own for pursuing computations. 

Then, we will investigate the new state sum model and the graded intersection model that we are interested in and write their formulas, which will be parametrised by certain intersection points in the symmetric power of the disc, graded in two different ways. In the last two steps, we show that these models lead to the same polynomial.

\subsection*{\bf Step 1-New state sum model}
Let us consider the Lagrangian submanifolds given by the geometric supports from the picture below:
\begin{center}
${\color{red} F_{\bar{i}} \subseteq C_{2n,(n-1)(N-1)}} \ \ \ \  \text{ and }\ \ \ \  {\color{dgreen} L_{\bar{i}}\subseteq C_{2n,(n-1)(N-1)}}.$
\begin{figure}[H]
\centering
\includegraphics[scale=0.4]{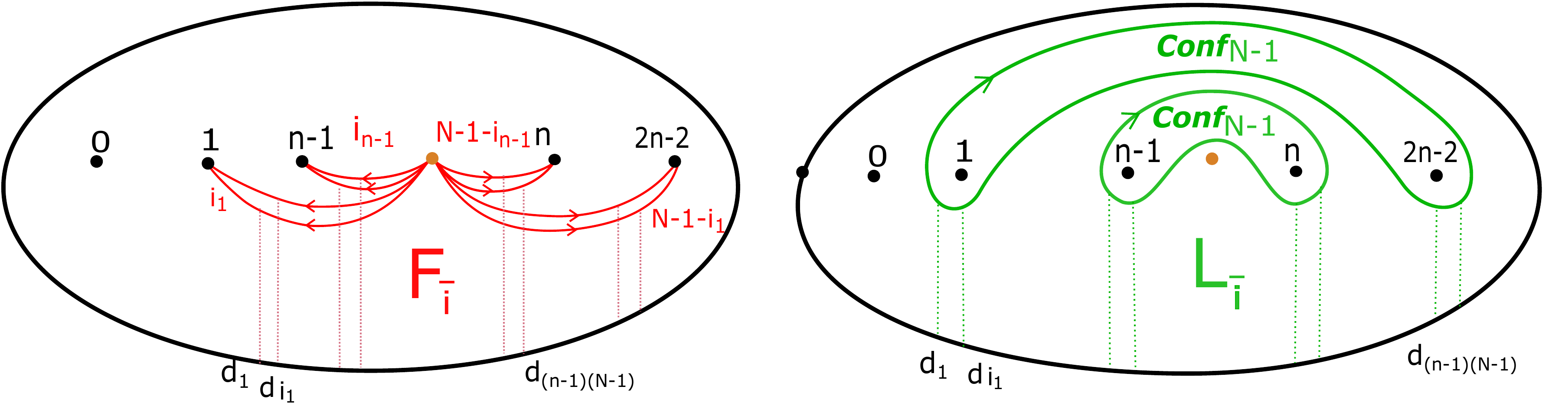}
\caption{State sum model}\label{Statesum}
\vspace{-25mm}
$$\hspace{20mm} {\color{red}\eta^{F_{\bar{i}}}} \hspace{53mm} {\color{dgreen}\eta^{L_{\bar{i}}} }\hspace{14mm}$$
\vspace{2mm}
\end{figure}
\end{center}

Then, let us consider the homology classes given by the lifts of these submanifolds in the associated covering space (constructed by the same procedure as the one from section \ref{homclasses}) and denote them by:
$${\color{red} \mathscr F_{\bar{i}} \in H_{2n,(n-1)(N-1)}} \ \ \text{ and }\ \  {\color{dgreen} \mathscr L_{\bar{i}}\in H^{\partial}_{2n,(n-1)(N-1)}}.$$

\begin{thm}[Unified embedded state sum model with new classes]\label{Thstate}
 Let us consider the following state sum for $\beta_n \in B_n$:
\begin{equation}\label{Fstate}
\begin{aligned}
\Lambda_N(\beta_n)(u,x,d):=u^{-w(\beta_n)} u^{-(n-1)} \sum_{i_1,...,i_{n-1}=0}^{N-1} & d^{-\sum_{k=1}^{n-1}i_k} \cdot \\
& \cdot \ll (\beta_{n} \cup {\mathbb I}_{n}){ \mathscr F_{\bar{i}}}, {\mathscr L_{\bar{i}}}\gg \ \in \Z[u^{\pm1},x^{\pm 1},d^{\pm 1}].
\end{aligned}
\end{equation}
Then $\Lambda_N$ gives the $N^{th}$ coloured Jones and $N^{th}$ coloured Alexander polynomials of an oriented link $L$ such that $L=\hat{\beta}_n$:
\begin{equation}
\begin{aligned}
&J_N(L,q)=\Lambda_N(\beta_n)|_{\gamma_{1,q,N-1}}\\
&\Phi_{N}(L,\lambda)=\Lambda_N(\beta_n)|_{\gamma_{1-N,\xi_N,\lambda}}.
\end{aligned}
\end{equation}
\end{thm}
\begin{proof}

We will show that $\Lambda_N(\beta_n)(u,x,d)=\Lambda'_N(\beta_n)(u,x,d)$. This will be done in two parts.
\begin{itemize}
\item[•]First, we study the action of the braid group on the first family of classes which are used in the definition of $\Lambda_N'(\beta_n)(u,x,d)$ and the ones which appear in the formula for $\Lambda_N(\beta_n)(u,x,d)$. We will see that these actions are very similar. 
\item[•] Secondly, we present a criterion stating the requirements for which family of dual classes one has to use in order to obtain the state sum $\Lambda_N'(\beta_n)(u,x,d)$. 
\item[•] Then, we prove that the second family of dual homology classes $\mathscr{L}_{\bar{i}}$ satisfies this property.\end{itemize} 
\begin{defn}
For a set of natural numbers $j_0,...,j_{2n-2}\in \N$ such that $$j_0+...+j_{2n-2}=(n-1)(N-1)$$ we denote by $\mathscr U_{j_0,...,j_{2n-2}}$, $\mathscr U'_{j_0,...,j_{2n-2}}$ the homology classs given by the lifts of the geometric supports from the picture below:
\begin{center}
$\mathscr U'_{j_0,...,j_{2n-2}} \in  H_{2n-1,(n-1)(N-1)} \ \ \ \ \ \ \ \ \ \ \ \ \ \ \ \  \ \ \ \ \ \ \ \ \ \mathscr U_{j_0,...,j_{2n-2}} \in  H_{2n,(n-1)(N-1)}$
\begin{figure}[H]
\centering
\includegraphics[scale=0.4]{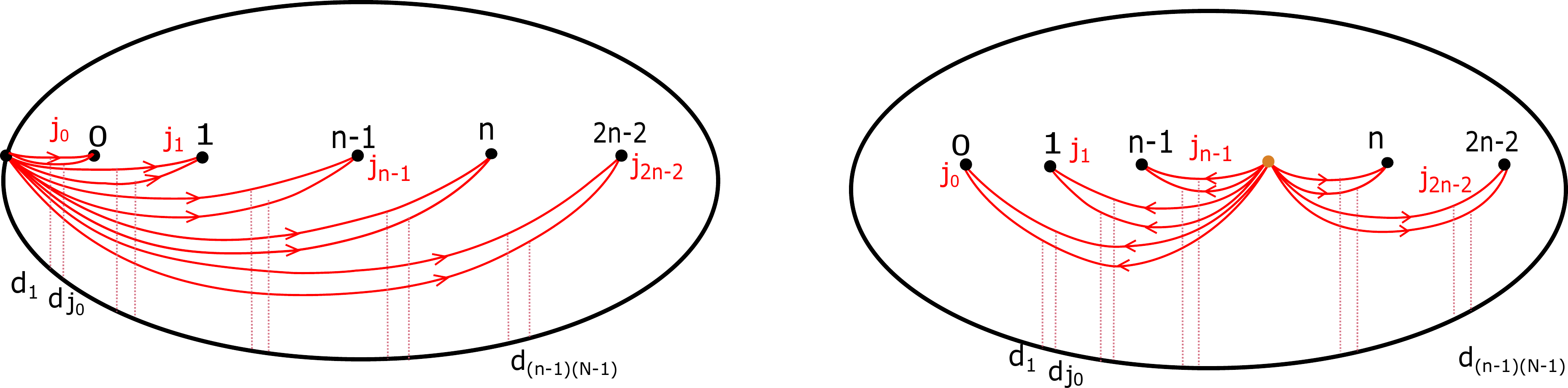}
\caption{Change of the first family of clases}\label{Statesum}
\vspace{-2mm}
\end{figure}
\end{center}
\end{defn}
Let us fix $\bar{i}=(i_1,...,i_{n-1})$. Using the above notation we have:
\begin{align*}
&\mathscr F'_{\bar{i}}=\mathscr U'_{0,i_1,...,i_{n-1},N-1-i_{n-1},...,N-1-i_{1}}\\
&\mathscr F_{\bar{i}}=\mathscr U_{0,i_1,...,i_{n-1},N-1-i_{n-1},...,N-1-i_{1}}.
\end{align*}
Looking at the action of braids on these classes and using the structure of these homology groups discussed in \cite{Cr1} (Subsection 7.5), we conclude that there are the following decompositions:
\begin{align}\label{eq:2}
&(\beta_{n} \cup {\mathbb I}_{n-1}) \mathscr F'_{\bar{i}}=\sum_{\substack{j_0,...,j_{n-1}=0 \\ j_0+...+j_{n-1}=(n-1)(N-1)}}^{N-1}\alpha'(j_0,...,j_{n-1}) \ \mathscr U'_{j_0,j_1,...,j_{n-1},N-1-i_{n-1},...,N-1-i_{1}}\\
&(\beta_{n} \cup {\mathbb I}_{n} )\mathscr F_{\bar{i}}=\sum_{\substack{j_0,...,j_{n-1}=0 \\ j_0+...+j_{n-1}=(n-1)(N-1)}}^{N-1}\alpha(j_0,...,j_{n-1}) \ \mathscr U_{j_0,j_1,...,j_{n-1},N-1-i_{n-1},...,N-1-i_{1}}
\end{align}
where $\alpha'(j_0,...,j_{n-1}),\alpha'(j_0,...,j_{n-1}) \in \Z[x^{\pm 1}, d^{\pm 1}]$.
Now we want to understand what is the relation between these coefficients and we will actually prove that:
\begin{equation}\label{eq:3}
\alpha'(j_0,...,j_{n-1})=\alpha(j_0,...,j_{n-1}), \forall \ j_0,...,j_{n-1}.
\end{equation} 
This relation can be seen by looking at the formula for these coefficients, which can be computed by pairing with a dual basis (as in \cite{CrM}, Section 3 and Proposition 7.6). This formula uses the red geometric supports, the paths to the boundary and the local system. Up to isotopy we can assume that the braid action is non-trivial just on a grey disc which goes around the first $n$-punctures.
Now, the red geometric support inside the grey disc is the the same in both pictures, the only change is part of the geometric support outside that disc. However, any path from the first picture corresponds to a path in the second picture which has the same winding around punctures and the same relative winding of the particles, and so any such paths (which compute the coefficients) are evaluated in the same way by the local system. 

We pass to the second part and investigate the intersections with the dual classes. 
In \cite{Cr2} (Section 6) we computed the following intersection:
\begin{equation}
\hspace{-3mm}\ll  \mathscr U_{j_0,...,j_{2n-2}}, {\mathscr L'_{\bar{i}}}\gg=\begin{cases}
1, \text{ if } (j_0,...,j_{2n-2})=(0,i_1,...,i_{n-1},N-1-i_{n-1},...,N-1-i_{1} )\\
0, \text{ if } j_0\neq 0 \text{ or } \exists \ k\in \{1,...,n-1\}: j_{k}\neq N-1-j_{2n-1-k}.
\end{cases}
\end{equation}
Moreover, following the argument from Step III and Step IV from Section 6 (\cite{Cr2}), we conclude the following criterion. 
\begin{lem}
If we have a family of homology classes $\mathscr W'_{\bar{i}}\in H^{\partial}_{2n-1,(n-1)(N-1)}$ such that:
\begin{equation}
\hspace{-3mm}\ll  \mathscr U'_{j_0,...,j_{2n-2}}, {\mathscr W'_{\bar{i}}}\gg=\begin{cases}
1, \text{ if } (j_0,...,j_{2n-2})=(0,i_1,...,i_{n-1},N-1-i_{n-1},...,N-1-i_{1} )\\
0, \text{ if } j_0\neq 0 \text{ or } \exists \ k\in \{1,...,n-1\}: j_{k}\neq N-1-j_{2n-1-k}
\end{cases}
\end{equation}
then they we can replace the classes $\mathscr L'_{\bar{i}}$ by $\mathscr W_{\bar{i}}$ in equation \eqref{Fstate'} without changing the value of the state sum. 
\end{lem}

On the other hand, we know that after we act with the braids and get the classes 
 $(\beta_{n} \cup {\mathbb I}_{n-1}) \mathscr F'_{\bar{i}}$ and $(\beta_{n} \cup {\mathbb I}_{n}) \mathscr F_{\bar{i}}$, they have the same coefficients in the decomposition with respect to the classes of the form $\mathscr U'_{j_0,...,j_{2n-2}}$ and $\mathscr U_{j_0,...,j_{2n-2}}$ respectively (as presented in the formulas \eqref{eq:2} and \eqref{eq:3}). This shows the following property.
\begin{lem}(Family of classes which lead to the same state sum)\label{Critsum}\\
Suppose that we have a family of homology classes $\mathscr L^W_{\bar{i}}\in H^{\partial}_{2n,(n-1)(N-1)}$ such that:
\begin{equation}
\hspace{-3mm}\ll  \mathscr U_{j_0,...,j_{2n-2}}, {\mathscr L^W_{\bar{i}}}\gg=\begin{cases}
1, \text{ if } (j_0,...,j_{2n-2})=(0,i_1,...,i_{n-1},N-1-i_{n-1},...,N-1-i_{1} )\\
0, \text{ if } j_0\neq 0 \text{ or } \exists k\in \{1,...,n-1\}: j_{k}\neq N-1-j_{2n-1-k}.
\end{cases}
\end{equation}
We consider the state sum using these dual classes:
\begin{equation}\label{Fstate}
\Lambda^W_N(\beta_n)(u,x,d):=u^{-w(\beta_n)} u^{-(n-1)} \sum_{i_1,...,i_{n-1}=0}^{N-1} d^{-\sum_{k=1}^{n-1}i_k}\ll (\beta_{n} \cup {\mathbb I}_{n}){ \mathscr F_{\bar{i}}}, {\mathscr W_{\bar{i}}}\gg.
\end{equation}
Then by the procedure of replacing the $\mathscr F'$-classes by $\mathscr F$-classes and the duals $\mathscr L'$ by $\mathscr L^W$ in $\Lambda'_N(\beta_n)$ we keep the same value of the state sum:
$$\Lambda'_N(\beta_n)=\Lambda^W_N(\beta_n).$$
\end{lem}
In the next part we show that the family $\mathscr L_{\bar{i}}$ has the above property. Let us fix $j_0,...,j_{2n-2}\in \N$ such that $j_0+...+j_{2n-2}=(n-1)(N-1)$. The pairing $\ll  \mathscr U_{j_0,...,j_{2n-2}}, {\mathscr L_{\bar{i}}}\gg$ is encoded by the geometric intersection in the configuration space between the submanifolds: $ U_{j_0,...,j_{2n-2}}$ and $L_{\bar{i}}$, as in the picture below. 
\begin{center}
\begin{figure}[H]
\centering
\includegraphics[scale=0.35]{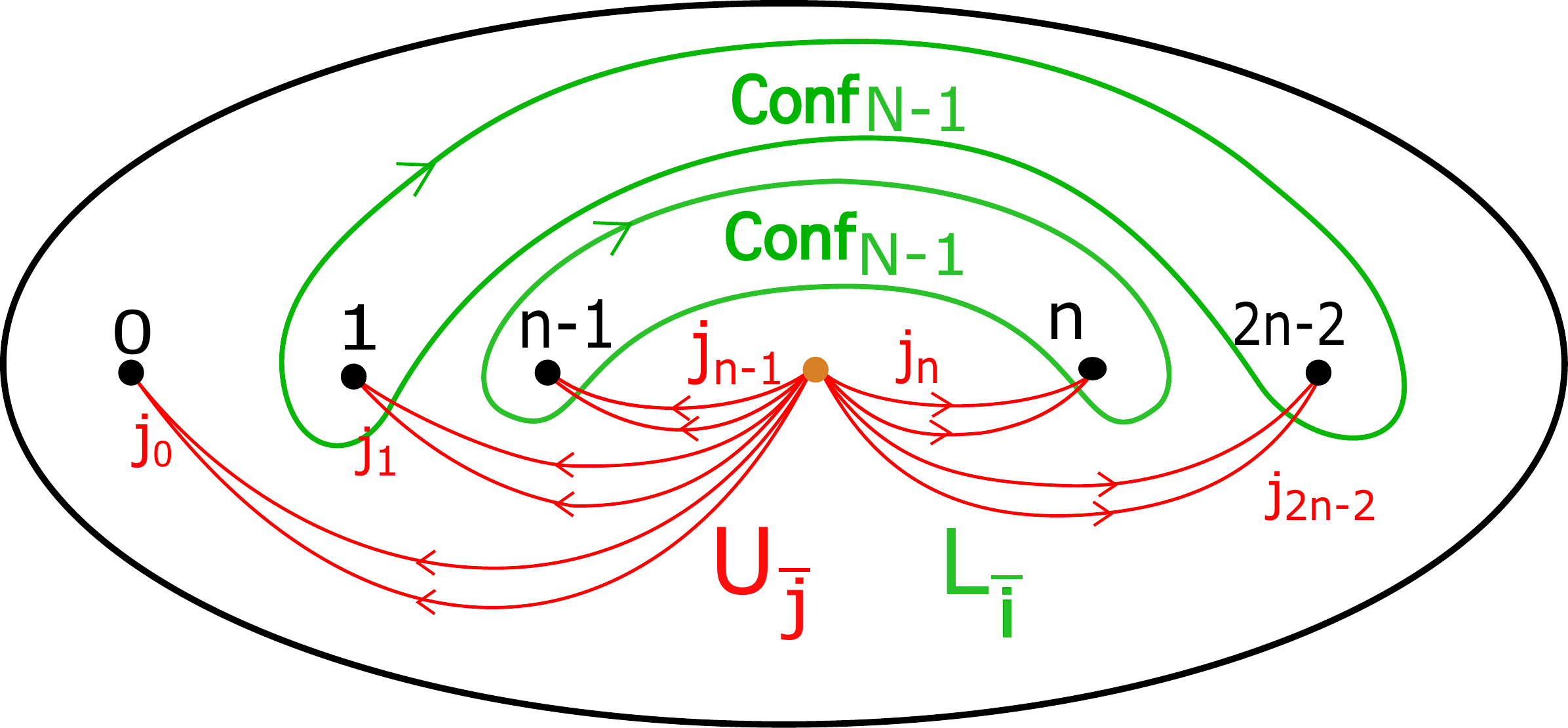}
\caption{Intersection between the geometric supports}
\end{figure}
\end{center}
\vspace{-10mm}
We will show that if 
$$U_{j_0,...,j_{2n-2}}\cap L_{\bar{i}}\neq \emptyset $$ then $j_0= 0$ and  $ j_{k}=N-1- j_{2n-1-k}, \forall k\in \{1,...,n-1\}.$
Let $\bar{w}=(w_0,...,w_{(n-1)(N-1)})\in U_{j_0,...,j_{2n-2}}\cap L_{\bar{i}}$. There are $N-1$ components of $w$ lying on the green circle around the punctures labeled by $(n-1,n)$ and also on the red support from the above picture. This means that:
$j_{n-1}+j_n\geq N-1$. Similarly, the $k^{th}$ circle intersects just the red curves which end in the punctures labeled by $(k,2n-1-k)$ and so we have that:
\begin{equation}
\begin{cases}
j_{1}+j_{2n-2}\geq N-1\\
...\\
j_{n-1}+j_n\geq N-1.\\
\end{cases}
\end{equation}
On the other hand, we know that $ j_0+...+j_{n-1}=(n-1)(N-1)$. Putting these relations together we conclude that:
$j_0= 0 \text{ and } j_{k}=N-1- j_{2n-1-k}, \forall k\in \{1,...,n-1\}.$
For the last part of the argument, it is a similar computation as the one from \cite{Cr2} (Proposition 6.4.3) to see that $$\ll  \mathscr U'_{j_0,...,j_{2n-2}}, {\mathscr L_{\bar{i}}}\gg=
1, \text{ if } (j_0,...,j_{2n-2})=(0,i_1,...,i_{n-1},N-1-i_{n-1},...,N-1-i_{1}).$$
This shows that the family of classes $\{\mathscr L_{\bar{i}}\}$ satisfies the requirements \eqref{Fstate} from Lemma \ref{Critsum} and we conclude the equality of the two state sums:
\begin{equation}\label{eq:1111}
\Lambda'_N(\beta_n)=\Lambda_N(\beta_n).
\end{equation}
\end{proof}
\subsection*{\bf Step 2- Formulas for the two intersection pairings}

We will prove that the graded intersection $\Gamma_N$ leads to the same polynomial as the state sum $\Lambda_N(\beta_n)(u,x,d)$, namely:
\begin{equation}
\Gamma_N(\beta_n)|_{y=-d}=\Lambda_N(\beta_n)(u,x,d), \forall \beta_n \in B_n.
\end{equation}
First of all, we choose the action of the braid group to be trivial outside a grey disc around the first punctures and also we choose two representatives of the classes  $\mathscr F_{\bar{i}}$ and $\mathscr L_{\bar{i}}$ which have the geometric supports that intersect inside the grey disc, in the left hand side of the punctured disc. 

To begin with we remind the formulas that give the above intersection pairings, following relations \eqref{inter}, \eqref{eq:111} and \eqref{eq:1111}:
\begin{align}
& \langle {(\beta_n \cup \mathbb I_{2n-1}})\mathscr S_{n}^N,\mathscr T_{n}^N  \rangle=\sum_{\bar{x} \in (\beta_n \cup \mathbb I_{2n-1})\mathscr S_{n}^N\cap\mathscr T_{n}^N } \epsilon_{x_1}\cdot...\cdot \epsilon_{x_{(n-1)(N-1)}}\cdot \left( \sum_{F \in \Co(\bar{x})} \varphi(l_{\bar{x},F}) \right)\\
& \ll (\beta_{n} \cup {\mathbb I}_{n} ){ \mathscr F_{\bar{i}}}, {\mathscr L_{\bar{i}}}\gg= \sum_{\bar{w}\in (\beta_{n} \cup {\mathbb I}_{n} ){ F_{\bar{i}}}  \cap L_{\bar{i}}} \epsilon_{w_1}\cdot...\cdot \epsilon_{w_{(n-1)(N-1)}}\cdot \varphi(l_{\bar{w}})
 \in \Z[x^{\pm1}, d^{\pm1}].
\end{align}
This means that the graded intersections have the following form:
\begin{align*}
\Gamma_N(\beta_n)(u,x,y,d)&= \ u^{-w(\beta_n)} u^{-(n-1)} (-y)^{-(n-1)(N-1)} \cdot \\
& \hspace{-2mm}\sum_{\bar{x}\in (\beta_n \cup \mathbb I_{2n-1})\mathscr S_{n}^N\cap\mathscr T_{n}^N} \epsilon_{x_1}\cdot...\cdot \epsilon_{x_{(n-1)(N-1)}}\cdot \left( \sum_{F \in \Co(\bar{x})} \varphi(l_{\bar{x},F}) \right)\\
 \Lambda_N(\beta_n)(u,x,d)&~=u^{-w(\beta_n)} u^{-(n-1)} \cdot \\ & \hspace{-2mm} \sum_{i_1,...,i_{n-1}=0}^{N-1} d^{-\sum_{k=1}^{n-1}i_k} \cdot  \sum_{\bar{w}\in (\beta_{n} \cup {\mathbb I}_{n} ){ F_{\bar{i}}}  \cap L_{\bar{i}}} \epsilon_{w_1}\cdot...\cdot \epsilon_{w_{(n-1)(N-1)}}\cdot \varphi(l_{\bar{w}})\in \Z[x^{\pm1},d^{\pm1}].
\end{align*}
We will prove that these two expressions become equal after we replace the $d$-coefficients from the second sum by the variable $-y$ as below:
\begin{equation}\label{toprove}
\hspace{-5mm}\begin{aligned}
&(-y)^{-(n-1)(N-1)} \sum_{\bar{x}\in (\beta_n \cup \mathbb I_{2n-1})\mathscr S_{n}^N\cap\mathscr T_{n}^N} \epsilon_{x_1}\cdot...\cdot \epsilon_{x_{(n-1)(N-1)}}\cdot \left( \sum_{F \in \Co(\bar{x})} \varphi(l_{\bar{x},F}) \right)=\\
&=\sum_{i_1,...,i_{n-1}=0}^{N-1} (-y)^{-\sum_{k=1}^{n-1}i_k} \cdot  \sum_{\bar{w}\in (\beta_{n} \cup {\mathbb I}_{n} ){ F_{\bar{i}}}  \cap L_{\bar{i}}} \epsilon_{w_1}\cdot...\cdot \epsilon_{w_{(n-1)(N-1)}}\cdot \varphi(l_{\bar{w}}) \in \Z[x^{\pm1},y^{\pm1}, d^{\pm1}].
\end{aligned}
\end{equation}
In the next part we investigate the terms that appear in the two previous sums. One first remark is that these pairings are given by intersections in different symmetric powers and also that the indexing sets provided by the intersection points between Lagrangian submanifolds are different. We introduce the following notations:
$$\mathscr I_{\beta_n}:=(\beta_n \cup \mathbb I_{2n-1})\mathscr S_{n}^N\cap\mathscr T_{n}^N$$
$$\mathscr K_{\beta_n}(\bar{i}):=(\beta_{n} \cup {\mathbb I}_{n} ){ F_{\bar{i}}}  \cap L_{\bar{i}}.$$
Then, the indexing sets for the two state sums are the following:
\begin{align*}
&I_{\Gamma_N(\beta_n)}:=\{(\bar{x},F)\mid \bar{x}\in \mathscr I_{\beta_n}, F \text{ colouring associated to the multiplicities of } \bar{x}\}\\
&I_{\Lambda_N(\beta_n)}:= \{(\bar{i},\bar{w})\mid \bar{i}=(i_1,...,i_{n-1}), 0 \leq i_1,...,i_{n-1} \leq N-1,  \bar{w} \in \mathscr K_{\beta_n}(\bar{i})\}.
\end{align*}
We will split the next part of the proof into two steps, as follows:
\begin{itemize}
\item[•] First, we want to create a bijective correspondence between these indexing sets:\\
$$\chi: I_{\Lambda_N(\beta_n)} \rightarrow I_{\Gamma_N(\beta_n)} .$$
\item[•]Secondly, once this is constructed, we will investigate the gradings of the loops associated to two points which correspond one another through the function $\chi$. 
\end{itemize}

\subsection*{\bf Step 3- Correspondence between the indexing sets using the intersection points}

\

Let $(\bar{i},\bar{w})\in I_{\Lambda_N(\beta_n)}$. In particular, we have that  $$\bar{w}=(w_1,...,w_{(n-1)(N-1)}) \in (\beta_{n} \cup {\mathbb I}_{n} ){ F_{\bar{i}}}  \cap L_{\bar{i}}.$$
This means that for any $k \in \{1,...,n-1\}$, there are $N-1$ components the point $\bar{w}$ which belong to the $k^{th}$ green circle from the support of $L_{\bar{i}}$. Now, if we compress the red curves from the support of ${ F_{\bar{i}}}$ which end in the same puncture we obtain the picture in the punctured disc which is on the left side of figure \ref{F}. On the right hand side we have the geometric support of the submanifolds $(\beta_n \cup \mathbb I_{2n-1}) \cs$ and $\ct$. 

We remark that those supports are exactly the same except the interior of the cream vertical band from the picture.
The $N-1$ components of $w$ belonging to the $k^{th}$ circle will transform through the compression process into a number of points, which we denote by $$y^k_1,...,y^k_{l(k)}.$$
Looking at the two pictures, we see that the intersection points between the red and the green supports are in one to one correspondence, since they are all in the exterior of the horizontal cream band. 
\begin{center}
\begin{figure}[H]
\centering
\includegraphics[scale=0.35]{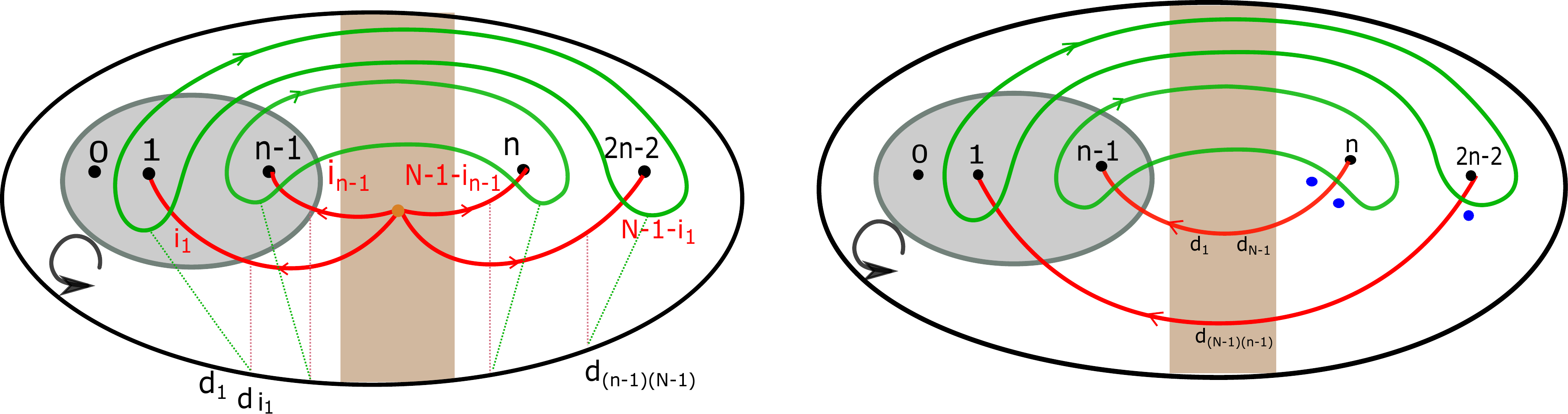}
\caption{Intersections between geometric supports}\label{F}
\end{figure}
\end{center}
\vspace{-10mm}
Then, we denote by $$x^k_1,...,x^k_{l(k)}$$ the points from the right picture associated to $y^k_1,...,y^k_{l(k)}$, which belong to the intersection $(\beta_{n} \cup {\mathbb I}_{2n-1} ){\cs}  \cap \ct$.
In order to define a point in the symmetric power of the disc using these points in the punctured disc, we have to prescribe their multiplicities. 
In order to do this, we go back and remember that the points $y^k_1,...,y^k_{l(k)}$ came from red curves with multiplicities, which we compressed during the above procedure. 
This means that for a fixed point $y^k_i$ there will be many points from the support of $(\beta_{n} \cup {\mathbb I}_{n} ){ F_{\bar{i}}}  \cap L_{\bar{i}}$ which get compressed to that point, and we denote this number by $m^k_i$.

We consider the point in the symmetric power given by the following elements in the punctured disc (with the associated multiplicities)
\begin{equation}
x^{(\bar{i},\bar{w})}:=\left((x^1_1,m^1_1),...,\left(x^{n-1}_{l(n-1)},m^{n-1}_{l(n-1)}\right)\right).
\end{equation}
In the next part we will define the colouring associated to this multipoint, in order to complete the definition of the function $\chi$. This time, we will use the multiplicities of the red curves which give $(\beta_{n} \cup {\mathbb I}_{n} ) F_{\bar{i}}$. We want to define a colouring as in definition \ref{loops}: $F:=(f^1,...,f^{n-1})$ where
$f^k:\{1,...,N-1\}\rightarrow\{1,...,l(k)\}$ for $k \in \{1,...,n-1\}$ such that: 
\begin{equation}
\card \left((f^k)^{-1}\{i\}\right)=m^k_i.
\end{equation} 
Let us fix $k \in \{1,...,N-1\}$. We look at the geometric support of the $k^{th}$ pair of red multicurves from the support of $(\beta_{n} \cup {\mathbb I}_{n} ) F_{\bar{i}}$  (this is the image through ($\beta_{n} \cup {\mathbb I}_{n}$) of the set of red curves which end in the punctures $k$ and $2n-1-k$ ). We begin from the puncture from left hand side of the disc where this pair ends. We start to follow this red geometric support and we encounter elements of $\bar{w}$ precisely in the places where the smoothing has elements of $\bar{y}$. 

This means that for such a point $y^i_j$ which belongs to the $k^{th}$ smoothened support, correspondingly we have $m_j^i $ points at the intersection between the supports of $(\beta_{n} \cup {\mathbb I}_{n} ) F_{\bar{i}}$ and $L_{\bar{i}}$.  If $y^i_j$ is in the left half of the disc, those $m_j^i $ points will lie on the $i_k$ red arcs (which all start and end in the same punctures) that give $y^i_j$ after the  compression procedure. 
We remark that we have maximum one point $y^i_j$ in the right hand side of the disc, and in this situation the $m_j^i =N-1-i_{k}$ components will be precisely the points of intersection between the $N-1-i_{k}$ red arcs which end in the puncture $2n-1-k$ and the green circle, as in the picture from the top ($1)$ and $2)$) of figure \ref{FF}.

We denote the set of such components of $\bar{w}$ which correspond to the points $y^i_j$ from the $k^{th}$ red support by: $$(c_1,...c_{N-1})$$
(counted by using the ordering the $N-1$ curves from the $k^{th}$ red support of $F_{\bar{i}}$ from left to right around their common intersection point).

 These points can be connected to $N-1$ points from the boundary of the disc following the paths that give $\eta^{F_{\bar{i}}}$. We think of these boundary points as being labeled by $\{1...,N-1\}$ from left to right. 
 
 Now, we remember that we can partition the set of points $(c_1,...c_{N-1})$ into $l(k)$ subsets, where each subset gets compressed the same point through the compression procedure. This means that we can colour the points $(c_1,...c_{N-1})$ by $l(k)$ colours, and looking at the correspondence to the points on the boundary, we have an induced colouring of the set $\{1,...,N-1\}$ into $l(k)$ colours. We define $f^k:\{1,...,N-1\}\rightarrow \{1,...,l(k)\}$ to be this function and consider the associated colouring: 
\begin{equation}
F^{(\bar{i},\bar{w})}:=\left(f^1,...,f^{n-1}\right).
\end{equation}
\begin{defn}(Correspondence between the indexing sets)
\

We consider the function $\chi: I_{\Lambda_N(\beta_n)} \rightarrow I_{\Gamma_N(\beta_n)}$ given by:
\begin{equation}
\chi\left((\bar{i},\bar{w})\right):=\left(x^{(\bar{i},\bar{w})},F^{(\bar{i},\bar{w})} \right).
\end{equation}
Further on, it is possible to reverse this procedure and starting from a point $(x,F)\in I_{\Gamma_N(\beta_n)}$ to construct an associated element in 
$I_{\Lambda_N(\beta_n)}$ and so we conclude that the function $\chi$ is a well defined bijection.
\end{defn}
\clearpage
\subsection*{\bf Step 4- Gradings in the two models}
In this last step we want to show that the gradings behave well with respect to the correspondence $\chi$. Following the notations for the indexing sets, we express relation \eqref{toprove} as below:
\begin{equation}
\hspace{-5mm}\begin{aligned}
&\sum_{(\bar{x},f)\in I_{\Gamma_N(\beta_n)}} \epsilon_{x_1}\cdot...\cdot \epsilon_{x_{(n-1)(N-1)}}\cdot \varphi(l_{\bar{x},F})=\\
&=  \sum_{(\bar{i},\bar{w})\in I_{\Lambda_N(\beta_n)}} \epsilon_{w_1}\cdot...\cdot \epsilon_{w_{(n-1)(N-1)}} \cdot  (-y)^{(n-1)(N-1)-\sum_{k=1}^{n-1}i_k} \cdot \varphi(l_{\bar{w}}).
\end{aligned}
\end{equation}
We will prove this relation. More specifically we will show that for $(\bar{i},\bar{w})\in I_{\Lambda_N(\beta_n)}$ we have:
\begin{equation}\label{C}
\epsilon_{w_1}\cdot...\cdot \epsilon_{w_{(n-1)(N-1)}} \cdot  (-y)^{(n-1)(N-1)-\sum_{k=1}^{n-1}i_k} \cdot \varphi(l_{\bar{w}})=\epsilon_{x_1}\cdot...\cdot \epsilon_{x_{(n-1)(N-1)}}\cdot \varphi(l_{\bar{x},F})
\end{equation}
where $(\bar{x},F)=\chi\left((\bar{i},\bar{w})\right)$.
Now we look at the associated loops in the symmetric power: $l_{\bar{w}}$ and $l_{\bar{x},F}$. We remark that each of these has in total  $$(N-1-i_1)+...+(N-1-i_{n-1})=(n-1)(N-1)-(i_1+...i_{n-1})$$
components in the punctured disc which are on the right hand side of it.
However, $l_{\bar{w}}$ does not have any winding, but the components of $l_{\bar{x},F}$ wind around the blue punctures. This means that we get a factor of  $$y^{(n-1)(N-1)-(i_1+...i_{n-1})}$$ when we evaluate the grading $\varphi(l_{\bar{x},F})$.
On the other hand, looking at the signs of the local orientations, we notice that the $(n-1)(N-1)-(i_1+...i_{n-1})$ components from 
$$(\beta_n \cup \mathbb I_{2n-1})\mathscr S_{n}^N\cap\mathscr T_{n}^N $$ which belong to the right hand side of the disc have negative orientations, and so we get an extra $(-1)^{(n-1)(N-1)}$ factor this way.

It remains to show that the contributions from the left hand side of the disc to the gradings of the two loops are equal.  
We start with the loop $l_{\bar{w}}$ and we look at its components in the punctured disc which start from the base points that are connected to the $k^{th}$ circle. The recipe for constructing this loop says that we have to start from these base points, go to the green circle and then continue on this green support towards the components of the intersection point $\bar{w}$ which belong to this circle, then go back following the red support, as in the first picture $1)$ from figure \ref{FF}. We remark that we can slide these base points vertically, preserving their ordering, and consider them fixed on the set of red curves as in the picture $1')$. This change will preserve the evaluation of the grading on this loop. From now on we use these base points, from picture $1')$ for the grading of the loops associated to the state sum coefficients.

Further on, we want to understand the effect of the compression procedure with respect to the grading of the loop. This procedure does not affect the first part of the construction where we start from the base points from picture \ref{FF} ($1')$ and $2')$) and then we follow the dotted curves to get to the green support. What changes is the construction of the second part of the loop, regarding the path that comes back from the intersection points towards the base points, following the red curves. 
Now, for $t\in \{1,...,n-1\}$ we imagine that we compress the set of $i_t$ and $N-1-i_{t}$ red curves from the support of $(\beta_{n} \cup {\mathbb I}_{n-1} ){ F_{\bar{i}}}  \cap L_{\bar{i}}$ towards the associated red curve from figure \ref{FF} $2')$. The base points will be sent to $d^t_1,...,d^{t}_{N-1}$ and the order will be preserved through this procedure. 

This shows that for any $k,l \in \{1,..,n-1\}$ with $k \neq l$, the contribution of relative grading (given by the intersection with the diagonal $\Delta$) coming from parts of the loops $l_{\bar{w}}$ and $l_{\bar{x},F}$ which go back towards base points belonging to the $k^{th}$ and $l^{th}$ red curves is the same in the two pictures. 
\clearpage
\begin{center}
$\Lambda_N(\beta_n) \hspace{60mm} \Gamma_N(\beta_n) $
$$(\beta_{n} \cup {\mathbb I}_{n} ){ F_{\bar{i}}}  \cap L_{\bar{i}} \hspace{40mm }(\beta_n \cup \mathbb I_{2n-1})\mathscr S_{n}^N\cap\mathscr T_{n}^N$$
\vspace{-10mm}
\begin{figure}[H]
\includegraphics[scale=0.4]{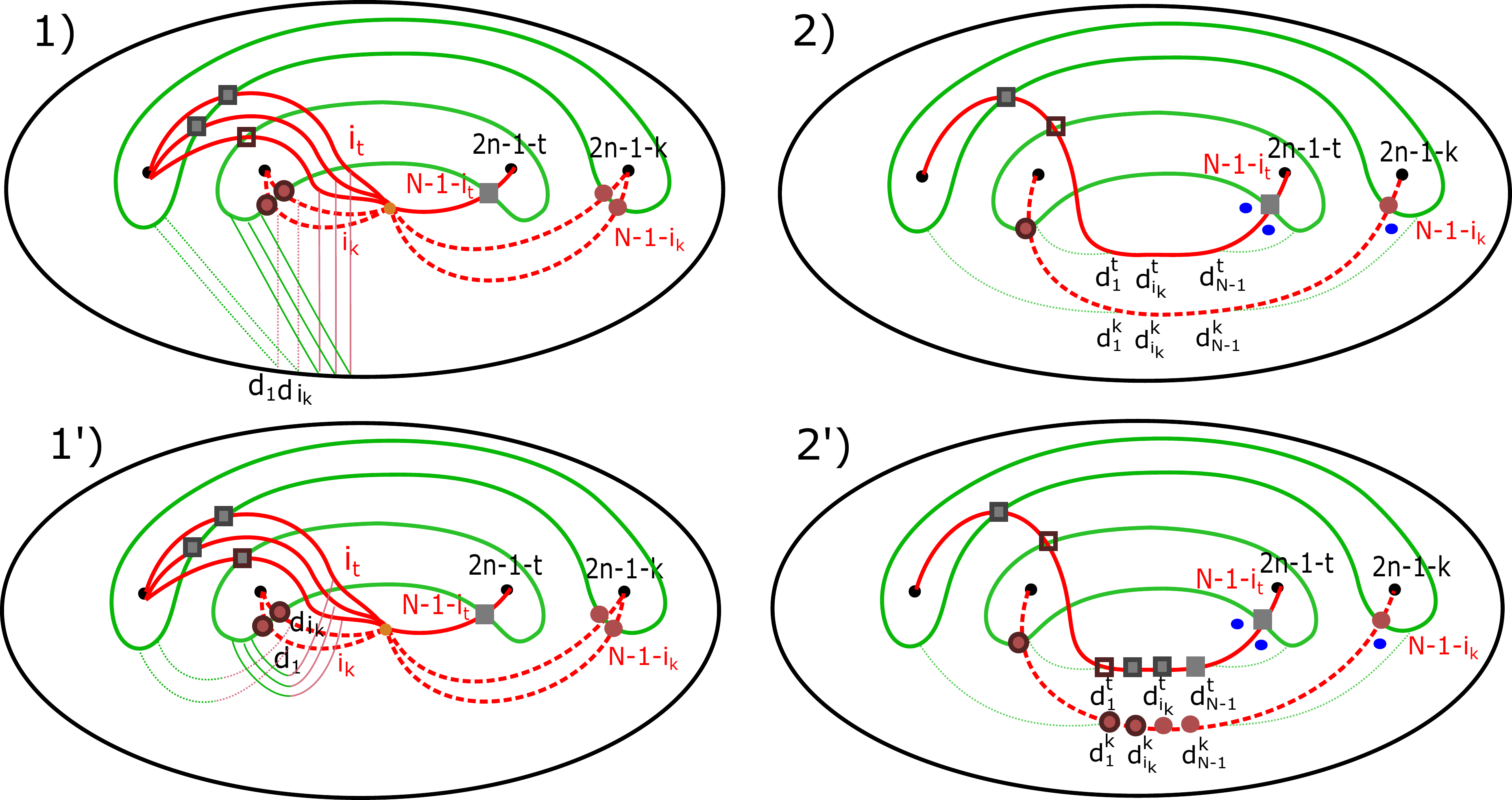}
\caption{Gradings in the two models}\label{FF}
\end{figure}
\end{center}

Now we fix any $t \in \{1,...,n-1\}$. The last subtlety to be discussed concerns the contribution to the relative grading coming from 
points which are chosen on the set of $i_t$ curves from figure $1')$ and the associated points with multiplicities from the compressed picture $2')$. 
More precisely, let us consider two different points $x^t_p$  and $x^t_q$ which are on the same red curve from the support $(\beta_n \cup \mathbb I_{2n-1})\cs$ from figure $2')$. Then they come from the compression of $m^t_p$ and $m^t_q$ points respectively, which belong to the set of red curves with multiplicities $i_t$ from the support $(\beta_{n} \cup {\mathbb I}_{n} ){ F_{\bar{i}}}  \cap L_{\bar{i}}$. We notice that all the $m^t_p$ points belong to the same circle and also the $m^t_q$ points are all on a different circle. Then following the construction of the loop $l_{\bar{w}}$, we have to go from these points to the base points from figure $1')$, which now lie on the red curves. The $m^t_p$ points are sent to a subset $S_p$ of cardinality $m^t_p$ of the set of $i_t$ base points which are on these red curves. Using the constructions of the colouring $F$ and the loop $l_{(\bar{x},F)}$ we see that in the compressed picture the multipoint $(x^t_p,m^t_p)$ is sent to the $m^t_p$ components of $\{d^t_1,...,d^t_{N-1}\}$ obtained from $S_p$ by compression. Similarly, the point $x^t_q$ with multiplicity $m^t_q$ will be connected through $l_{(\bar{x},F)}$ to a set of base points from $\{d^t_1,...,d^t_{N-1}\}$ obtained from $S_q$ by compression. 
We remark that different choices of fixing $m^t_p$ points from the $i_t$ points which get compressed to $x^t_p$ and $m^t_q$ points from the $i_q$ points which get compressed to $x^t_q$ induce different colourings on the right picture. 

With this we see that the loops $l_{(\bar{x},F)}$ and $l_{\bar{w}}$ induce the same permutation of the set of ordered base points. The difference between these loops comes from the fact that $l_{\bar{w}}$ does not intersect the diagonal, but it winds around it and it has a contribution to the relative winding coming from the choice of the above $m^t_p$ and $m^t_q$ points on the $i_t$ red curves. On the other hand the loop $l_{(\bar{x},F)}$ encodes directly this relative winding induced by the multiplicity $m^t_p$ and $m^t_q$ points by the fact that it intersects the diagonal thanks to the colouring. We obtain that overall the two loops have equal contributions to the relative winding. 

In other words, the grading on the left hand side changes according to the relative winding number induced by such choices of $m^t_p$ and $m^t_q$ points, which on the right hand side will be encoded precisely by the intersection of the loop $l_{(\bar{x},F)}$ with the diagonal, which will be created by the colouring. 

We conclude that the contribution of the gradings coming from the parts of the loops $l_{(\bar{x},F)}$ and $l_{(\bar{w})}$ from the left hand side of the disc is the same, and so relation \eqref{C} holds.

This concludes the proof of the intersection model in the symmetric power. 
\section{Unified intersection model for Jones and Alexander polynomials}\label{S:5'}
This part concerns the particular case given by $N=2$, which shows a unified model for the original Jones and Alexande polynomials. In this situation, the ambient manifold is the $(n-1)$-symmetric power of the  $(3n-1)$-punctured disc. On the other hand, we remark that the submanifolds $\mathscr S_{n}^2$ and $\mathscr T_{n}^2$ are given directly by the products of the red curves and green circles respectively, as in the picture below. This means that these submanifolds do not intersect the diagonal of the symmetric power and so they belong to the configuration space $Conf_{n-1}(\mathscr D_{3n-1})$.  
$$ \ \ \ \ \ \mathscr S_{n}^2  \ \ \ \text{ and } \ \ \  \mathscr T_{n}^2$$
\vspace{-10mm}
\begin{center}
\begin{figure}[H]
\centering
\includegraphics[scale=0.4]{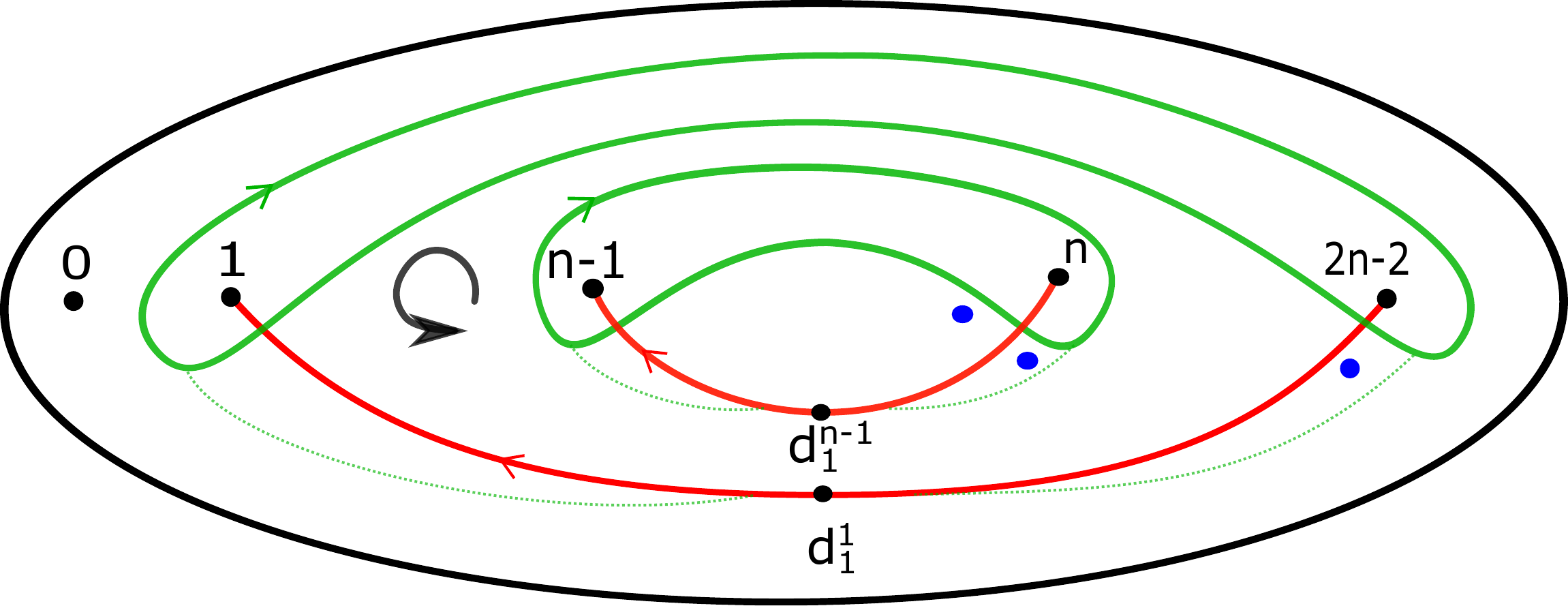}
\hspace{100mm}$\Large{Conf_{n-1}(\mathscr D_{3n-1})}$

\caption{Jones and Alexander polynomials}\label{F}
\end{figure}
\end{center}
Also, on each red curve there is just one base point, so for each intersection point $\bar{x}$ there is a unique associated colouring $F_{\bar{x}}$. We conclude that there are no choices for colourings in this case.
This means that for each intersection point $\bar{x}\in (\beta_n \cup \mathbb I_{2n-1})\mathscr S_{n}^N\cap\mathscr T_{n}^N $ we have a unique loop in the symmetric power $l_{\bar{x},F_{\bar{x}}}$, which is constructed following the recipe from definition \ref{defloop}, and we denote it by $l_{\bar{x}}$.

We notice that these loops $l_{\bar{x}}$ are included in the configuration space, since they do not intersect the diagonal $\Delta$, and so we can compute the intersection pairing directly in the configuration space $Conf_{n-1}(\mathscr D_{3n-1})$.
\begin{thm}(Jones and Alexander polynomials from two Lagrangians in a configuration space)\label{THEOREM'}\\
 We have the intersection pairing, given by the graded intersection points between the two Lagrangians in the configuration space of $(n-1)-$particles in the punctured disc:
\begin{equation}
\Gamma_2(\beta_n)(u,x,y,d):=u^{-w(\beta_n)} u^{-(n-1)} (-y)^{-(n-1)(N-1)} \sum_{\bar{x} \in (\beta_n \cup \mathbb I_{2n-1})\mathscr S_{n}^2\cap\mathscr T_{n}^2 } \epsilon(\bar{x}) \cdot \phi(l_{\bar{x}}).
\end{equation}
This specialises to the Jones and Alexander polynomials of the closure of the braid, as below:
\begin{equation}\label{eqC:3}
\begin{aligned}
&J(L,q)= \Gamma_2(\beta_n)|_{u=q;x=q^{2}, d=q^{-2}}\\
&\Delta(L,x)=\Gamma_2(\beta_n) |_{u=x^{-\frac{1}{2}};y=1;d=-1}. \end{aligned}
\end{equation}
\end{thm}
\subsection{Alexander polynomial from the configuration space}\label{SS}
We end this section with some remarks concerning the difference between the Jones polynomial and Alexander polynomials seen through this model. 
Following the above relation, we see that the variable $y$ gets specialised to $1$ for the Alexander polynomial. This means that in this case the blue punctures play no role, so we can remove them and then we have a pairing in the punctured disc $\mathscr D_{2n-1}$. Also, the relative grading counted by $d$ gets evaluated to $-1$. Following \cite{Cr2}, we can count the relative grading by $-d$, if we use the intersection of the submanifolds in the configuration space rather than the product of local intersections in the punctured disc, using the discussion from Remark 3.3.3 and the going back to the model from Proposition 3.3.2 from \cite{Cr2}.

\begin{defn}
For $n\in \N$ we consider the morphism given by:
\begin{equation}
\begin{aligned}
& \varphi: \pi_1(Conf_{n-1}(\mathscr D_{2n-1}))\rightarrow  \Z\\
& \hspace{42.5mm}\langle x \rangle\\
&\varphi(\sigma)=x^{W(\sigma,\{0,...,n-1\})-W(\sigma,\{n,...,2n-2\})} 
\end{aligned}
\end{equation}
where $W(\sigma,\{p_1,...,p_i\})$ is the winding number of the loop $\sigma$ around the punctures $p_1,...,p_i$. 

This comes from the local system $\phi$ defined in subsection \ref{G}, but here $\varphi$ counts just the intersection with the diagonals $\Delta^{n}_P$ and $\Delta^{-(n-1)}_P$ of the symmetric power (which are related to the punctures of the puntured disc), and does not record the relative winding given by the intersections with the diagonal $\Delta$, as in picture \ref{FF}.
\end{defn}

$$\hspace{-60mm}(\beta_n \cup \mathbb I_{n-1})\mathscr S_{n}^2\cap\mathscr T_{n}^2$$
\vspace{-8mm}
\begin{center}
\begin{figure}[H]
\centering
\includegraphics[scale=0.34]{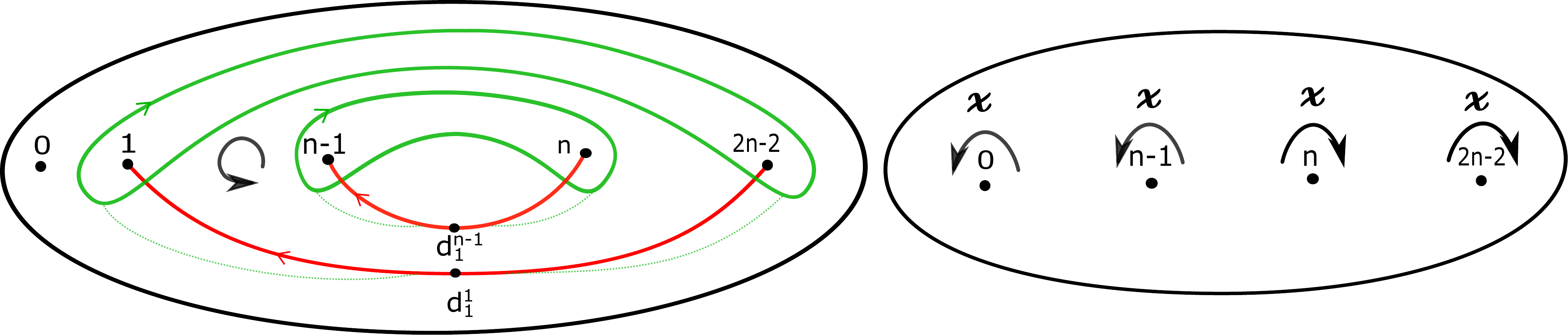}
\hspace{95mm}$\Large{Conf_{n-1}(\mathscr D_{2n-1})}$ \hspace{65mm} $\varphi$

\caption{Alexander polynomial from the intersection pairing \hspace{10mm} Local system \hspace{18mm}}\label{FF}
\end{figure}
\end{center}
In the next formula, for $\bar{z} \in (\beta_n \cup \mathbb I_{n-1})\mathscr S_{n}^2\cap\mathscr T_{n}^2$ we denote by $\alpha_{\bar{z}}$ the sign of the local intersection between the submanifolds $(\beta_n \cup \mathbb I_{n-1})\mathscr S_{n}^2$ and $\mathscr T_{n}^2$ of the configuration space at the point $\bar{z}$. Putting all these remarks together, we conclude following model for the Alexander polynomial.
 \begin{coro}(Alexander polynomial from a Lagrangian intersection) \label{THEOREM'''}\\
  The Alexander polynomial of a closure of a braid $\beta_n \in B_n$ is given by the intersection between the following two Lagrangians in the configuration space of $(n-1)-$particles in the punctured disc:
\begin{equation}
\Delta(L,x)=x^{-\frac{w(\beta_n)+(n-1)}{2}} \sum_{\bar{z} \in (\beta_n \cup \mathbb I_{n-1})\mathscr S_{n}^2\cap\mathscr T_{n}^2 } \alpha_{\bar{z}} \cdot \varphi(l_{\bar{z}}).
\end{equation}
\end{coro}

\section{Examples: Trefoil knot and $8_{19}$ knot}\label{S:5}
\subsection{Trefoil knot}
Let $T$ be the trefoil knot, which is the closure of $\sigma^3$ for $\sigma \in B_2$. We fix $N=2$, so we will work directly in the punctured disc $\mathscr D_5$. Then, we look at the two submanifolds and compute the gradings as in the picture below: $$ (\sigma^3 \cup \mathbb I_{3})\mathscr S_{2}^2\cap\mathscr T_{2}^2$$
\vspace{-5mm}
\begin{center}
\begin{figure}[H]
\centering
\includegraphics[scale=0.6]{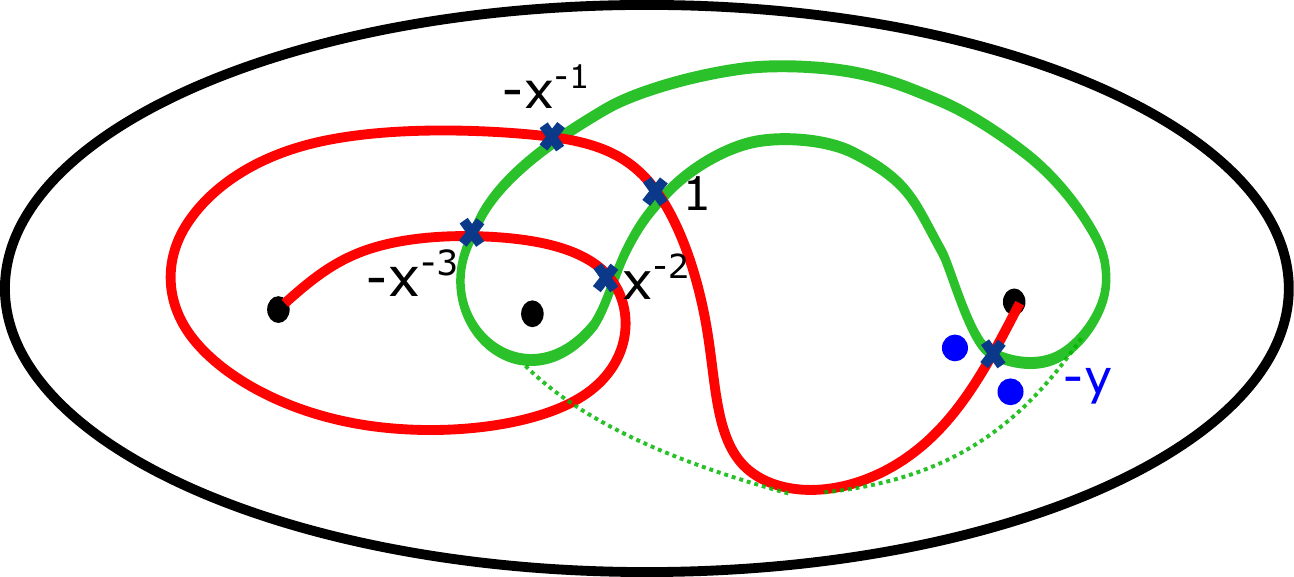}
\caption{Trefoil knot}
\end{figure}
\end{center}
\vspace{-5mm}
Then, the intersection pairing has the following expression:
\begin{equation}
\begin{aligned}
&\Gamma_2(\sigma^3)(u,x,y,d):=
u^{-4}(-y)^{-1}\left( -x^{-3}+x^{-2}-x^{-1}+1-y\right).
\end{aligned}
\end{equation}
Following relation \eqref{eqC:3}, for the Jones polynomial we have to compute the specialisation given by $u=q, \ x=q^{2}, \ y=-q^{-2}\ d=q^{-2}$ and we get the normalised version of this invariant:
\begin{equation}
\begin{aligned}
J(T,q)=\Gamma_2(\sigma^3)(u,x,y,d)|_{\psi_{1,q,1}}= -q^{-8}+q^{-2}+q^{-6}.
\end{aligned}
\end{equation}
For the Alexander polynomial, we have the specialisation  $u={\xi_2}^{-\lambda}, \ x={\xi_2}^{2\lambda}, y=-{\xi_2}^{-2}=1 \ d={\xi_2}^{-2}=-1$ and we obtain:
\begin{equation}
\begin{aligned}
\Phi_2(T,\lambda)=\Gamma_2(\sigma^3)(u,x,d)|_{\psi_{-1,\xi_2=i,\lambda}}={\xi_2}^{2\lambda}-1+{\xi_2}^{-2\lambda}.
\end{aligned}
\end{equation}
Or, if we write it in terms of $x=\xi^{2 \lambda}$ we get the usual form of the Alexander polynomial:
\begin{equation}
\begin{aligned}
\Delta(T,x)=x-1+x^{-1}.
\end{aligned}
\end{equation}
\subsection{Knot $8_{19}$}
In this part we will compute the intersection pairing for the knot $8_{19}$, seen as the closure of the braid $$\beta:=\sigma^3_1 \sigma_2 \sigma^3_1 \sigma_2 \in B_3.$$ In this case, we have to work in the configuration space of $(2-1)(3-1)=2$ points in the punctured disc with $8$ punctures. We have drawn below the geometric supports of $$(\beta \cup \mathbb I_{5})\mathscr S_{3}^2 \ \ \ \ \ \ \ \ \text { and } \ \ \ \ \ \ \ \   \mathscr T_{3}^2.$$
\begin{center}
\begin{figure}[H]
\centering
\includegraphics[scale=0.5]{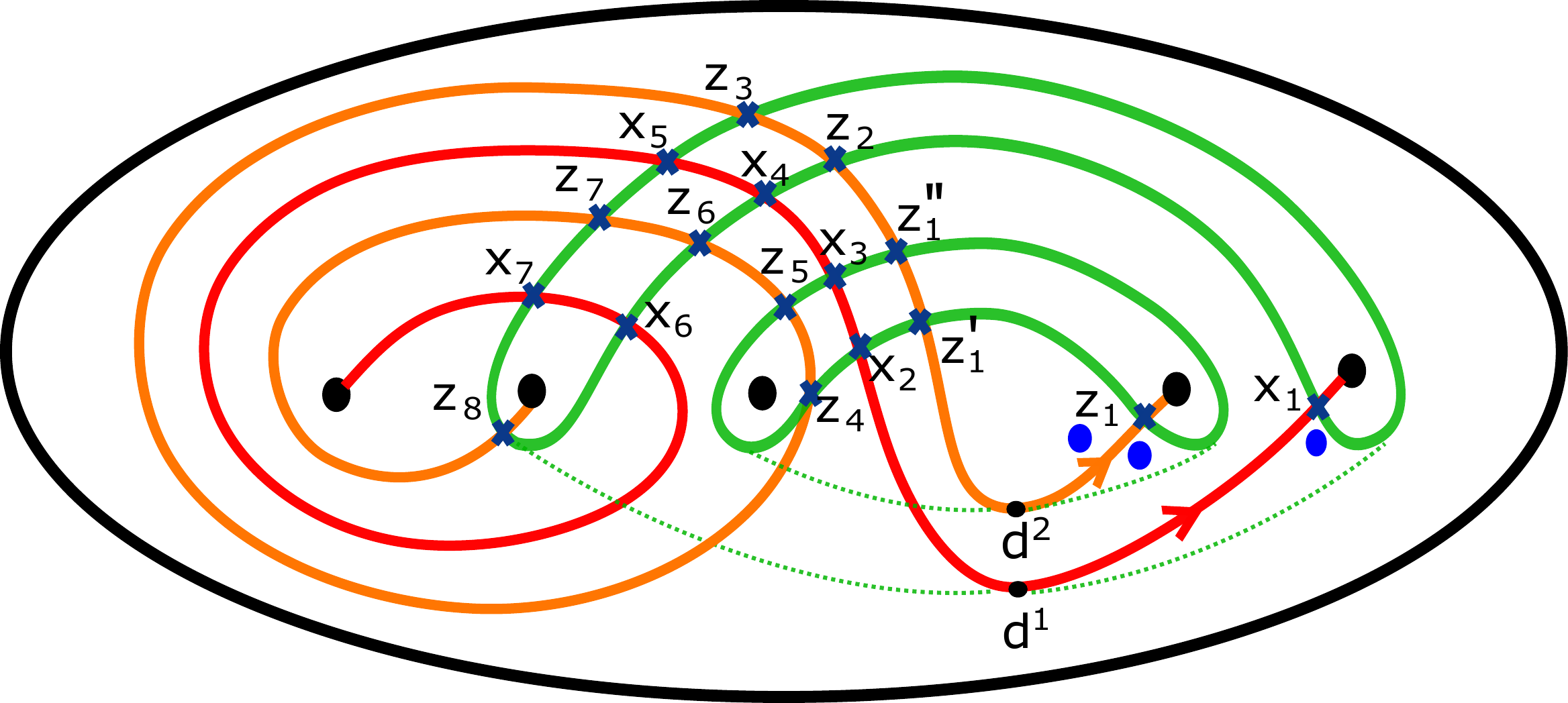}
\vspace{-5mm}

\hspace{100mm}$\Large{Conf_2(\mathscr D_8)}$

\vspace{5mm}
\caption{Knot $8_{19}$}
\end{figure}
\end{center}
We have the following set of generators which belong to the intersection between the submanifolds $(\beta \cup \mathbb I_{5})\mathscr S_{3}^2$ and $\mathscr T_{3}^2$ in the configuration space of two points in the punctured disc, which carry the associated gradings from the tabel below.
\begin{center}
\label{tabel}
{\renewcommand{\arraystretch}{1.7}\begin{tabular}{ | c | c | c | c | c | }
 \hline
$(x_1,z_1)$ & $(x_1,z^{'}_1)$ & $(x_1,z^{"}_1)$ & $(x_1,z_4)$ & $(x_1,z_5)$  \\ 
\hline  
$y^2$ & $-y$ & $yx^{-1}$ & $-yx^{-3}$ & $yx^{-4}$ \\ 
\hline
\hline                                    
$(x_4,z_1)$ & $(x_4,z^{'}_1)$ & $(x_4,z^{"}_1)$ & $(x_4,z_4)$ & $(x_4,z_5)$  \\ 
\hline  
$-yx^{-1}$ & $x^{-1}$ & $-x^{-2}$ & $d^{-2}x^{-4}$ & $-d^{-2}x^{-5}$ \\ 
\hline
\hline                                    
$(x_5,z_1)$ & $(x_5,z^{'}_1)$ & $(x_5,z^{"}_1)$ & $(x_5,z_4)$ & $(x_5,z_5)$  \\ 
\hline  
$yx^{-2}$ & $-x^{-2}$ & $x^{-3}$ & $-d^{-2}x^{-5}$ & $d^{-2}x^{-6}$ \\ 
\hline
\hline                                    
$(x_6,z_1)$ & $(x_6,z^{'}_1)$ & $(x_6,z^{"}_1)$ & $(x_6,z_4)$ & $(x_6,z_5)$  \\ 
\hline  
$-yx^{-3}$ & $x^{-3}$ & $-x^{-4}$ & $d^{-2}x^{-6}$ & $-d^{-2}x^{-7}$ \\ 
\hline
\hline                                    
$(x_7,z_1)$ & $(x_7,z^{'}_1)$ & $(x_7,z^{"}_1)$ & $(x_7,z_4)$ & $(x_7,z_5)$  \\ 
\hline  
$yx^{-4}$ & $-x^{-4}$ & $x^{-5}$ & $-d^{-2}x^{-7}$ & $d^{-2}x^{-8}$ \\ 
\hline
\hline                                    
$(x_2,z_2)$ & $(x_2,z_3)$ & $(x_2,z_6)$ & $(x_2,z_7)$ & $(x_2,z_8)$  \\ 
\hline  
$d^{-1}x^{-1}$ & $-d^{-1}x^{-2}$ & $d^{-1}x^{-4}$ & $-d^{-1}x^{-5}$ & $d^{-1}x^{-6}$ \\ 
\hline
\hline                                    
$(x_3,z_2)$ & $(x_3,z_3)$ & $(x_3,z_6)$ & $(x_3,z_7)$ & $(x_3,z_8)$  \\ 
\hline  
-$d^{-1}x^{-2}$ & $d^{-1}x^{-3}$ &-$d^{-1}x^{-5}$ & $d^{-1}x^{-6}$ & -$d^{-1}x^{-7}$ \\ 
\hline
\hline                                    
\end{tabular}}
\end{center}
Then the intersection pairing has the following form:
\begin{equation}
\begin{aligned}
\Gamma_2(\beta)(u,x,y,d)=u^{-10}y^{-2}\cdot (&y^2-y+yx^{-1}-yx^{-3}+yx^{-4}\\ 
&-yx^{-1}+x^{-1}-x^{-2}+d^{-2}x^{-4}-d^{-2}x^{-5} \\ 
&+yx^{-2}-x^{-2}+x^{-3}-d^{-2}x^{-5}+d^{-2}x^{-6} \\ 
&-yx^{-3}+x^{-3}-x^{-4}+d^{-2}x^{-6}-d^{-2}x^{-7} \\ 
&yx^{-4}-x^{-4}+x^{-5}-d^{-2}x^{-7}+d^{-2}x^{-8} \\ 
&d^{-1}x^{-1}-d^{-1}x^{-2}+d^{-1}x^{-4}-d^{-1}x^{-5}+d^{-1}x^{-6} \\ 
&-d^{-1}x^{-2}+d^{-1}x^{-3}-d^{-1}x^{-5}+d^{-1}x^{-6}-d^{-1}x^{-7}). 
\end{aligned}
\end{equation}
Replacing $y=-d$, we have the form in $3$-variables as below:
\begin{equation}
\begin{aligned}
\Gamma_2(\beta)(u,x,-d,d)=u^{-10} \Bigl( &1+d^{-1}+(d^{-2}+d^{-3})x^{-1}+ \\ 
&+(-d^{-1}-2d^{-2}-2d^{-3})x^{-2}+ \\ 
&+(2d^{-1}+2d^{-2}+d^{-3})x^{-3} \\ 
&+(-2d^{-1}-2d^{-2}+d^{-3}+d^{-4})x^{-4}\\ 
&+(d^{-2}-2d^{-3}-2d^{-4})x^{-5} \\ 
&+(-2d^{-3}+2d^{-4})x^{-6}- \\ 
&-(d^{-3}+2d^{-4})x^{-7}+d^{-4}x^{-8}\Bigr). 
\end{aligned}
\end{equation}
Specialising $u=x^{\frac{1}{2}},d=x^{-1},x=q^{2}$, we get the Jones polynomial of the knot $8_{19}$:
\begin{equation}
\begin{aligned}
\Gamma_2(\beta)(u,x,y,d)|_{\psi_{1,q,1}}&=q^{-10}(1-q^{-6}+q^{4})=\\
&=q^{-10}-q^{-16}+q^{-6}. 
\end{aligned}
\end{equation}
Also, if we specialise $u=x^{-\frac{1}{2}},d=-1$, we get the Alexander polynomial of the knot $8_{19}$:
\begin{equation}
\begin{aligned}
\Gamma_2(\beta)(u,x,y,d)|_{u=x^{-\frac{1}{2}},y=-d=1}&=x^{5}(x^{-2}-x^{-3}+x^{-5}-x^{-7}+x^{-8})=\\
&=x^{3}-x^{2}+1-x^{-2}+x^{-3}. 
\end{aligned}
\end{equation}

\noindent {\itshape University of Geneva, Section de mathématiques, 
Rue du Conseil-Général 7-9, Geneva, CH 1205 Switzerland}

\noindent {\tt Cristina.Palmer-Anghel@unige.ch}

\noindent \href{http://www.cristinaanghel.ro/}{www.cristinaanghel.ro}

\end{document}